%% file: paper.tex
\documentclass[12pt]{article}

\usepackage{a4wide,amssymb,amsfonts,amsmath,amsthm,csquotes,paralist,hyperref,cleveref}
\usepackage[shortlabels]{enumitem}
\usepackage[normalem]{ulem}
\usepackage{mathtools}
\usepackage{math_def}

\newtheorem{theorem}{Theorem}[section]

\newtheorem{corollary}[theorem]{Corollary}
\newtheorem{lemma}[theorem]{Lemma}

\newtheorem{remark}{Remark}[section]
\newtheorem{example}{Example}[section]

\usepackage[most]{tcolorbox}
\makeatletter
\let\vrb@comment    \comment
\let\comment        \@undefined
\let\vrb@endcomment \endcomment
\def\tcb@@end@v@@d{%
  \vrb@endcomment%
  \tcb@layer@dec%
}
\def\tcb@@capture@v@@d{%
  \let\endtcolorbox\tcb@@end@v@@d%
  \vrb@comment%
}
\makeatother

\newtcolorbox{blackbox}{boxrule=.5pt, sharp corners, colback=white, colframe=black, notitle, breakable, enhanced}

\usepackage[final]{ifdraft}
\ifoptionfinal{ \newcommand{\mycheck}[1]{#1} }{ \newcommand{\mycheck}[1]{{\begingroup\color{green!70!black}#1\endgroup}} }

\newcommand{\Asb}{{\mycheck{U}}} 
\newcommand{\As}{{\mycheck{\overbar{U}}}} 
\newcommand{\Asp}{{\mycheck{G}}} 
\newcommand{\Ase}{{\mycheck{X}}} 

\newcommand{\Ts}{{\mycheck{V}}} 
\newcommand{\Tsp}{{\mycheck{H}}} 
\newcommand{\Tse}{{\mycheck{Y}}} 

\newcommand{\Bp}{{\mycheck{B}}_{\!\circ}} 
\newcommand{\Bb}{{\mycheck{B}}} 
\newcommand{\B}{{\mycheck{\overbar{B}}}} 
\newcommand{\bb}{{\mycheck{b}}} 
\renewcommand{\b}{{\mycheck{\overline{b}}}} 

\newcommand{\cL}{\mycheck{\mathcal{L}}} 
\newcommand{\Lis}{\mycheck{\cL_{\mathrm{is}}}} 

\newcommand{\crefl}{\labelcref}		

\begin{document}

\setcounter{page}{1}

\title{A unified framework for the analysis, numerical approximation and model reduction of linear operator equations,\\ Part I: Well-posedness in space and time}
\author{Moritz Feuerle\footnote{Institute for Numerical Mathematics, Ulm University, Germany (moritz.feuerle@uni-ulm.de)}, 
    Richard L\"{o}scher\footnote{Institute of Applied Mathematics, TU Graz, Austria (loescher@math.tugraz.at)}, 
    Olaf Steinbach\footnote{Institute of Applied Mathematics, TU Graz, Austria (o.steinbach@tugraz.at)}, 
    Karsten Urban\footnote{Institute for Numerical Mathematics, Ulm University, Germany (karsten.urban@uni-ulm.de)}}
\maketitle

\begin{abstract}
    We present a unified framework to construct well-posed formulations for large classes of linear operator equations including elliptic, parabolic and hyperbolic partial differential equations. This general approach incorporates known weak variational formulations as well as novel space-time variational forms of the hyperbolic wave equation. The main concept is completion and extension of operators starting from the strong form of the problem.

    This paper lays the theoretical foundation for a unified approach towards numerical approximation methods and also model reduction of parameterized linear operator equations which will be the subject of the following parts.
\end{abstract}

\input{sections/introduction.tex}

\input{sections/analysis.tex}

\input{sections/applications.tex}

\input{sections/conclusion.tex}

\section*{Acknowledgments}
M. Feuerle acknowledges support by the Federal State of Baden-Württemberg within the Cooperative Research Training Group \enquote{Data Science and Analytics}.

\bibliographystyle{abbrv}
\bibliography{literature_converted_to_BibTeX.bib}

\end{document}

%% file: sections/introduction.tex
\section{Introduction}\label{ch: introduction}
Linear operator equations can, e.g., be derived from partial differential and integral equations. Typical examples include elliptic, parabolic and hyperbolic second order partial differential equations (PDEs) as well as boundary integral equations. The analysis of well-posedness (existence, uniqueness and stability of solutions) as well as the construction and investigation of numerical approximation methods are usually done problem-specific and different for the above mentioned classes of problems.

The aim of this paper is to provide a unified framework for the analysis, numerical approximation and model reduction of linear operator equations. In particular we aim at deriving a general approach which allows us to construct well-posed and optimally stable\footnote{In the sense that the norms of the operator and its inverse are equal to $1$.} variational formulations for rather general linear operator equations. This is the scope of part I of this paper series. The forthcoming part will then concentrate on deriving corresponding discretizations and numerical approximations which will benefit from the general framework constructed in this first part. Finally, the close relation between the approximation error and the residual will then also allow us to construct a quite general approach towards model reduction of \emph{parameterized} linear operator equations. 

In such a general framework we are able to cover stationary elliptic as well as time-dependent parabolic and hyperbolic problems. Elliptic problems are formulated in Sobolev and the time-dependent problems in Lebesgue-Bochner spaces using space-time variational formulations. The latter ones are well established for parabolic problems, but less is known for transport, wave and Schrödinger-type problems. The framework presented in this paper allows for well-posed variational formulations for all those equations. We shall detail the general approach to several examples, which will also include known weak and ultra-weak formulations. To the very best of our knowledge, this is a novel framework. It is, however, related to \cite{DahmeHSW2012}, which focuses on ultra-weak formulations. We will comment on similarities and differences throughout the paper. Our approach also extends existing results in a unified manner, recovering also well-known formulations for elliptic and parabolic problems, while also allowing for well-posed formulations for hyperbolic problems.

The remainder of this paper is organized as follows. After collecting some notation and guiding examples in the sequel of this section,  Section \ref{ch: theory} is devoted to the presentation of the theoretical foundation of our general framework. The main technical ingredient is completion and extension of operators to be presented in \S\ref{subsec: compl_ext}. We show applications to the Poisson problem, the heat equation and the wave equation in Section \ref{ch: applications}. The paper ends with some conclusions and an outlook in Section \ref{sec:conclusions}.

\subsection{Notation and basic facts}
All spaces are assumed to be vector spaces over a common field $\F\in\set{\R,\C}$. Thereby, for a normed vector space $\Ase$, we denote by $\norm{\cdot}_\Ase$ its norm and by $\Ase'$ its dual space ($\F=\R$) or anti-dual space ($\F=\C$). If $\Ase$ is an inner product space, we denote its inner product by $\inner{\cdot,\cdot}_\Ase$. Hence, the inner product is a bilinear ($\F=\R$) or sesquilinear\footnote{linear in the first and anti-linear in the second argument} form ($\F=\C$). Further, we denote by $\pairing{\cdot,\cdot}{\Ase}:\Ase'\times \Ase\ra\F$ the duality pairing (or evaluation map), given by $\pairing{f,x}{\Ase} \deq f(x)$ for $f\in \Ase'$ and $x\in \Ase$, which is also  a bilinear or sesquilinear form, respectively. For two Banach spaces $\Ase,\Tse$, we denote the space of all \emph{linear and bounded operators} by 
$\cL(\Ase,\Tse') \deq \set{B:\Ase\ra \Tse'\,:\, B \text{ is linear and } \norm{B}_{\cL(\Ase,\Tse')} < \infty}$, 
with $\norm{B}_{\cL(\Ase,\Tse')} \deq \sup\limits_{x\in \Ase\setminus\set{0}} \frac{\norm{Bx}_{\Tse'}}{\norm{x}_\Ase},$
and the space of all \emph{linear isomorphisms} by
\begin{equation*}
    \Lis(\Ase,\Tse') \deq \set{B \in \cL(\Ase,\Tse')\,:\, B \text{ bijective and } B^{-1} \in \cL(\Tse',\Ase)}.\label{eq: Lis}
\end{equation*}
As usual, we call $B \in \cL(\Ase,\Tse')$ \emph{isometric}, if $\norm{B}_{\cL(\Ase,\Tse')} = 1$. Given $B\in \cL(\Ase,\Tse')$ and $f\in \Tse'$, seeking for an unknown $x\in\Ase$, we call $Bx=f$ in $\Tse'$ interpreted as $\pairing{Bx,y}{\Tse} = \pairing{f,y}{\Tse}$ for all $ y\in\Tse$, a \emph{linear operator equation}. We call it \emph{well-posed} (in the sense of Hadamard, \cite{Hadamard}) if $B\in\Lis(\Ase,\Tse')$ and \emph{optimally stable} if $B$ is isometric. Further, for a Hilbert space $\Tsp$, we denote by $R_{\Tsp} \in \Lis(\Tsp,\Tsp')$ the isometric \emph{Riesz operator} given by
\begin{equation}\label{eq: riesz operator}
    \pairing{R_{\Tsp}x,y}{\Tsp} \deq \inner{x,y}_\Tsp\quad\forall x,y\in\Tsp.
\end{equation}
In addition, it holds $\inner{\cdot,\cdot}_{\Tsp'} = \pairing{\cdot,R_{\Tsp}^{-1}\cdot}{\Tsp}$, $\norm{\cdot}_{\Tsp} = \norm{R_{\Tsp}\cdot}_{\Tsp'}$ and $\norm{\cdot}_{\Tsp'} = \norm{R_{\Tsp}^{-1}\cdot}_{\Tsp}$.
For two normed vector spaces $\Ase,\Tse$ with $\Ase\subseteq \Tse$, we call $\Ase$ \emph{continuously embedded} in $\Tse$ (abbreviated \enquote{$\Ase\hookrightarrow \Tse$}) if there exists an \emph{embedding constant} $C<\infty$, such that $\norm{x}_\Tse \leq C\norm{x}_\Ase$ for all $x\in \Ase$.
Further, we denote by $\Ase\subseteq_d \Tse$, that $\Ase$ is dense in $\Tse$, and by $\Ase \hookrightarrow_d \Tse$, that $\Ase \subseteq_d \Tse$ and $\Ase \hookrightarrow \Tse$.
We collect some well-known facts.
\begin{remark}\label{rm: continuous embedding}
    If $\Ase\subseteq_d \Tse\hookrightarrow_d Z$, then $\Ase \subseteq_d Z$ and $\Ase \hookrightarrow Z$ if $\Ase\hookrightarrow \Tse\hookrightarrow Z$.
\end{remark}

\begin{remark}[{\cite[\S 5 Remark 3]{Brezi2011}}]\label{rm: gelfand}
    Let $\Ase \hookrightarrow_d \Tsp$ be Banach spaces with embedding constant $C$.
    \begin{compactenum}[(i)]
        \item It holds $\Tsp' \hookrightarrow \Ase'$ with embedding constant $C$ and $\Tsp' \hookrightarrow_d \Ase'$, if $\Ase$ is reflexive.
        \item Let $\Tsp$ be a Hilbert space. Identifying $\Tsp \cong \Tsp'$ using the Riesz operator, we get
              \begin{equation}\label{eq: gelfand}
                  \Ase \hookrightarrow_d \Tsp \cong \Tsp' \hookrightarrow \Ase',
              \end{equation}
              as well as $\norm{x}_\Tsp \leq C\norm{x}_\Ase$, $x\in \Ase$, $\norm{x}_{\Ase'} \leq C\norm{x}_\Tsp$, $x\in \Tsp$ 
              and for the duality paring holds
              \begin{align}
                  \pairing{f,x}{\Ase} = \inner{f,x}_\Tsp \qquad\forall f\in \Tsp,\ \forall x\in \Ase.\label{eq: gelfand pairing = product}
              \end{align}
              Eq.\ \eqref{eq: gelfand} is usually called a \emph{Gelfand triple}, denoted by $(\Ase,\Tsp,\Ase')$.
              Note: If $\Ase$ is a Hilbert space itself, there exists the natural Riesz isomorphism between $\Ase$ and $\Ase'$, but it is not viewed as the identity map, i.e., $\Ase\not\cong \Ase'$. Instead, we use the Riesz isomorphism between $\Tsp$ and $\Tsp'$ as the identity map to identify $\Tsp \cong \Tsp'$. To make things clear, for $x\in \Ase$, the dual norm is to be understood as $\norm{x}_{\Ase'} = \norm{R_\Tsp x}_{\Ase'} = \sup\limits_{y\in X}\frac{\inner{x,y}_\Tsp}{\norm{x}_\Ase}$ and not as $\norm{R_\Ase x}_{\Ase'}= \sup\limits_{y\in X}\frac{\inner{x,y}_\Ase}{\norm{x}_\Ase}$.
    \end{compactenum}
\end{remark}

\subsubsection*{Function spaces}
Let $I\deq (0,T)$, $0<T<\infty$ be a time interval, $\Omega\subset\R^d$, $d\in\N$, be an open bounded spatial domain with Lipschitz boundary for $d\geq 2$ and let $Q \deq I\times\Omega$ denote the space-time domain.
We denote by $C^{k;\ell}(Q)$, $k, \ell\in\N_0$,  the set of all $C(Q)$ functions where all partial derivatives w.r.t.\ $t\in I$ up to order $k$ and all partial derivatives w.r.t.\ $x\in\Omega$ up to order $\ell$ are in $C(Q)$, define $C_0(\overline{\Omega})\deq\set{u\in C(\overline{\Omega}):u\vert_{\partial\Omega}=0}$ and we denote the space of continuous functions with zero boundary in space and $k^{\text{th}}$-order zero initial conditions in time by
\begin{equation*}
        C_{0,;0}^{(k)}(\overbar{Q}) \deq \set{u\in C(I\times\overline{\Omega})\cap C^{k;0}([0,T)\times \Omega)\,:\, u\vert_{I\times\partial\Omega}=0,\  \partial^\alpha_tu(0)=0,\ \alpha=0,..., k}.
\end{equation*}
For function spaces in time on $\bar{I}$, we shall always denote homogeneous initial conditions $u(0)=0$ by the index \enquote{$0,$} and homogeneous terminal conditions $u(T)=0$ by \enquote{$,0$}. As usual, $C^\infty_0(\Omega)$ is the space of compactly supported functions on $C^\infty(\Omega)$.

For Hilbert spaces $\Ase$ and a domain $\Omega$ (which could also be a time interval $I$ or a space-time domain $Q$), we denote by $L^2(\Omega;\Ase)$, $H^1(\Omega;\Ase)$ and $H^1_0(\Omega;\Ase)$ the  Lebesgue-Bochner, Sobolev-Bochner and zero-trace Sobolev-Bochner Hilbert space (see, e.g., \cite{HytoeNVW2016} for an introduction to Bochner spaces), endowed with the inner product $\inner{u,v}_{L^2(\Omega;\Ase)} \deq \bintegral{\Omega}{}{\,\inner{u(x),v(x)}_{\Ase}}{x}$, $\inner{u,v}_{H^1(\Omega;\Ase)} \deq \inner{u,v}_{L^2(\Omega;\Ase)} + \inner{\nabla u,\nabla v}_{L^2(\Omega;\Ase)}$ and $\inner{u,v}_{H^1_0(\Omega;\Ase)} \deq \inner{\nabla u,\nabla v}_{L^2(\Omega;\Ase)}$, respectively. Further, we denote by $H^{-1}(\Omega;\Ase) \equiv [H^1_0(\Omega;\Ase)]'$ the dual space of $H^1_0(\Omega;\Ase)$ and define the Hilbert space $H^\Delta(\Omega;\Ase) \deq \set{u\in H^1(\Omega;\Ase):\Delta u \in L^2(\Omega;\Ase)}$ endowed with $\inner{u,v}_{H^\Delta(\Omega;\Ase)} \deq \inner{u,v}_{H^1(\Omega;\Ase)}+\inner{\Delta u,\Delta v}_{L^2(\Omega;\Ase)}$. 
In addition, for a time interval $I$, it holds $H^1(I;\Ase)\hookrightarrow C(\overbar{I};\Ase)$ by \cite[Proposition II.5.11]{BoyerF2013}, i.e., point evaluations $u(t) \in \Ase$, $t\in\overbar{I}$ are well-defined for $u\in H^1(I;\Ase)$, thus we can define the Hilbert spaces
\begin{gather*}
    H^1_{0,}(I;\Ase) \deq \set{u\in H^1(I;\Ase)\,:\,u(0)=0},\quad
    H^1_{,0}(I;\Ase) \deq \set{u\in H^1(I;\Ase)\,:\,u(T)=0},
\end{gather*}
endowed with the inner product $\inner{\cdot,\cdot}_{H^1_0(I;\Ase)}$.
We write $L^2(\Omega)$, $H^1(\Omega)$, $H^1_0(\Omega)$, $H^{-1}(\Omega)$, $H^\Delta(\Omega)$, $H^1_{0,}(I)$ and $H^1_{,0}(I)$ for $L^2(\Omega;\R)$, $H^1(\Omega;\R)$, $H^1_0(\Omega;\R)$, $H^{-1}(\Omega;\R)$, $H^\Delta(\Omega;\R)$, $H^1_{0,}(I;\R)$ and $H^1_{,0}(I;\R)$, respectively.
Finally, for a second Hilbert space $\Tse$, we equip the vector space $\Ase\times\Tse$ with the inner product $\inner{\vec{u},\vec{v}}_{\Ase\times\Tse} \deq \inner{u_1,v_1}_{\Ase}+\inner{u_2,v_2}_{\Tse}$,
thus forming again a Hilbert space.

\subsection{Examples}\label{ch: introduction examples}
We collect examples which we will reconsider in the course of this paper. All these will be expressed (in strong form) as $\Bp u = f$  (or $\Bp \vec{u} = \vec{f}$), including initial and/or boundary conditions.

\begin{example}[Poisson equation]\label{ex: elliptic problem}
    For $f:\Omega\ra\R$ and $u:\overline{\Omega} \ra \R$, the Poisson equation takes the form $-\Delta u=f$ in $\Omega$ and $u=0$ on $\partial\Omega$.
\end{example}

\begin{example}[Heat equation]\label{ex: heat equation}
    For $f:Q\ra\R$ and $u:\overbar{Q}\ra\R$, the heat equation is expressed by $u_t -\Delta_x u = f$ in $Q$, $u(0) = 0$ in $\Omega$ and $u=0$ on $I\times \partial\Omega$.
\end{example}

\begin{example}[Wave equation]\label{ex: wave equation}
    For $f:Q\ra\R$ and $u:\overbar{Q}\ra\R$, the wave equation reads $u_{tt} -\Delta_x u =  f$ in $Q$, $u(0)= \partial_t u(0) = 0$ in $\Omega$ and $u = 0$ on $I\times\partial\Omega$.
\end{example}

\begin{example}[First-order in time wave equation]\label{ex: FO wave equation}
    For $\vec{f}: Q\ra\R^2$ and $\vec{u} :\overbar{Q}\ra\R^2$, the first-order in time formulation of the wave equation reads $\partial_t \vec{u} + A_\circ\vec{u} = \vec{f}$ in $Q$, $\vec{u}(0)=0$ in $\Omega$ and $\vec{u}=0$ on $I\times\partial\Omega$, 
    where $A_\circ :=  \begin{psmallmatrix}0 & -\Id\\ -\Delta_x & 0\end{psmallmatrix}$ and $\Id$ denoting the identity operator. For $\vec{f}=(0,f)$, this is equivalent to \cref{ex: wave equation} with $\vec{u}$ and $u$ being related by $\vec{u}=(u,\partial_t u)$.
\end{example}

\begin{remark}\label{rm:general elliptic operaotrs possible}
    For the sake of simplicity, we only consider homogeneous initial and boundary conditions and the Laplace operator $-\Delta_x$ in space, but the above examples can be extend directly to (i) inhomogeneous initial and/or boundary conditions $u(0)=g_0$ in $\Omega$, $\partial_t u(0) = g_2$ in $\Omega$ and $u=g_3$ on $\partial\Omega$ or $u=g_3$ on $I\times\partial\Omega$ using standard homogenization techniques; and (ii) a uniformly elliptic and bounded (time variant) spatial differential operator
    $A_\circ(t)u(x) \deq - \nabla_x \cdot \parens[\big]{\uline{A}(t,x) \nabla_x u(x)} + \uline{b}(t,x)\cdot \nabla_x u(x) + \uline{c}(t,x) u(x)$     
    replacing $-\Delta_x$, with $\uline{A}:Q\ra\R^{d\times d}$, $\uline{b}:Q\ra\R^d$ and $\uline{c}:Q\ra\R$ sufficiently smooth.
\end{remark}

%% file: sections/analysis.tex
\section{Well-posed optimally stable weak formulations}\label{ch: theory}

We present the announced general framework for optimally stable weak formulations for linear operator equations, which we will then apply to the above mentioned example problems.

\subsection{Operator equations}\label{ch: operator equations}

We are now going to formulate the class of linear operator equations that we will consider in the sequel.

\subsubsection*{Classical form}
We consider linear operator equations on function spaces. Typically, those operators are defined pointwise by their mapping properties (i.e., a differential or integral operator) and corresponding initial and/or boundary conditions. This means in particular that those kind of conditions are included in the definition of the operator. In order to distinguish between the pointwise interpretation of the operator and the subsequent variational form, we add the subscript  \enquote{${}^\circ$} to indicate the pointwise form. Hence, we start by $\Bp : D(\Bp)\to C(\Omega)$, where $\Omega\subset\R^d$ is the domain of the primitive variables on which the seeked function is defined (which may be replaced by a space-time cylinder $Q$, see \S\ref{ch: introduction examples}). Here,
\begin{equation*}
    D(\Bp) := \{ u\in C(\Omega):\, \Bp u\in C(\Omega)\}
\end{equation*}
is the \emph{classical domain} of the operator. Then, for $f\in C(\Omega)$, the classical/pointwise formulation of an operator equation amounts seeking $u^*_\circ \in \domain(\Bp)$ satisfying 
\begin{equation}\label{eq:analysis:classical problem}
    \Bp u^*_\circ = f\text{ in }\Omega,
    \quad\text{i.e.,}\quad
    \Bp u^*_\circ(x) = f(x)\,\, \forall x\in\Omega. 
\end{equation}

\begin{example}\label{ex: classical domains}
    With the operator equations in \S\ref{ch: introduction examples} their classical domains read:
    \begin{compactenum}[(i)]
        \item $\domain(\Bp) = C^2(\Omega) \cap C_0(\overline{\Omega})$ for the Poisson equation in \cref{ex: elliptic problem};
        \item $D(\Bp)=C^{1;2}(Q)\cap C_{0,;0}^{(0)}(\overbar{Q})$ for the heat equation in \cref{ex: heat equation};
        \item $D(\Bp)= C^{2;2}(Q)\cap C_{0,;0}^{(1)}(\overbar{Q})$ for the wave equation in \cref{ex: wave equation};
        \item $D(\Bp)= C^{1;2}(Q;\R^2)\cap C_{0,;0}^{(0)}(\overbar{Q};\R^2)$ for the first-order in time wave equation in \cref{ex: FO wave equation}.
    \end{compactenum}
\end{example}

\subsubsection*{Variational formulation}
The classical form of an operator equation can only be expected to be well-posed in exceptional cases (depending on $\Omega$ and the data). Hence, we consider a variational formulation and stress the fact that such a variational form is not unique. Thus, we introduce an abstract framework for variational formulations and detail it for the examples mentioned above.

For a general framework, let $\Asb$ be a Banach space (called the \emph{trial} or \emph{ansatz space}), $\Ts$ be a reflexive Banach space (called the \emph{test space}) and $\Tsp$ be a Hilbert space, such that $(\Ts,\Tsp,\Ts')$ forms a Gelfand triple, i.e.\ 
$\Ts\hookrightarrow_d\Tsp \cong \Tsp' \hookrightarrow_d\Ts'$. 
Then, a variational operator reads $\Bb: \domain(\Bb) \subseteq \Asb \ra \Ts'$, where $\Bb$ is a (possibly unbounded) linear operator defined on a linear subspace $\domain(\Bb) \subseteq \Asb$, called the \emph{domain} of $\Bb$, see \cite[\S2.6]{Brezi2011}. The operator $\Bb$ is called \emph{bounded}, if $\domain(\Bb) = \Asb$ and $\Bb\in \cL(\Asb,\Ts')$.
For given $f\in\Ts'$, an abstract variational formulation amounts seeking $u^*\in \domain(\Bb)$ satisfying the operator equation
\begin{equation}\label{eq:operator-equation}
    \Bb u^* = f\text{ in $\Ts'$},
    \quad\text{i.e.,}\quad
    \pairing{\Bb u^*, v}{\Ts} = \pairing{f, v}{\Ts}\quad \forall v\in\Ts.
\end{equation}
As $\Bb$ is neither assumed to be bounded nor bijective, \eqref{eq:operator-equation} is in general not well-posed.

\subsubsection*{Examples}\label{ch: variational formulations examples}
We indicate some possible variational formulations for the examples introduced in \S\ref{ch: introduction examples}. Generally speaking, to derive a variational formulation \eqref{eq:operator-equation} of the classical form \eqref{eq:analysis:classical problem}, select a Hilbert space $\Tsp$, typically $\Tsp \deq L^2(\Omega)$, multiply by a test function $v \in C^\infty_0(\Omega)$ and, in case of $\Tsp=L^2(\Omega)$, integrate over $\Omega$, leading to
\begin{equation*}
    \inner{\Bp u,v}_{\Tsp} = \inner{f,v}_{\Tsp}\qquad \forall u\in\domain(\Bp),\ \forall v\in C^\infty_0(\Omega).
\end{equation*}
After possibly applying integration by parts on the left-hand side, this gives rise to a variational operator $\Bb$, for wich the domain $\domain(\Bb)$ and codomain $\Ts'$ are now to be determined. Starting by the codomain, $\Ts$ has to be a reflexive Banach space such that $\Ts \hookrightarrow_d \Tsp$ while the relation $\inner{\Bp u,v}_{\Tsp} = \pairing{\Bb u,v}{\Ts}$ has to hold for all $v\in\Ts$ instead of for all $v\in C^\infty_0(\Omega)$ (and all $u\in\domain(\Bp)$). Of course, this leaves some room for the choice of $\Ts$, although $\Ts$ is typically chosen as the largest of these spaces. Next, for the domain $\domain(\Bb)$ we need $\domain(\Bp) \subseteq \domain(\Bb)$ and it has to impose all initial and boundary conditions that have not already been imposed implicitly during the construction of $\Bb$ using integration by parts. Although $\domain(\Bp)$ does not need to be a Banach space itself, it has to be embedded (densely for \S\ref{ch: gelfand trial}) in some Banach space $\Asb$. By these assumptions, it becomes evident, that every solution $u^*_\circ$ of the classical formulation \eqref{eq:analysis:classical problem} also solves the variational formulation \eqref{eq:operator-equation}.
Regarding our naming convention, we call a variational formulation (of a 2nd order operator) \emph{strong}, \emph{weak} or \emph{ultra-weak} (in a variable), if it was derived by applying none, one or two integrations by parts (with respect to this variable).

\par\smallskip
To get a better hold on this procedure, consider Poisson's equation in \cref{ex: elliptic problem} and let $u\in \domain(\Bp)$ with $\domain(\Bp)$ given in \cref{ex: classical domains}. Multiplying by $v \in C^\infty_0(\Omega)$ and integrating over $\Omega$ gives $\inner{\Bp u,v}_{L^2(\Omega)} = \inner{-\Delta_x u,v}_{L^2(\Omega)} $ while further applying one or two integrations by parts (recalling that $u_{|\partial\Omega}=v_{|\partial\Omega}=0$) yield $ \inner{\Bp u,v}_{\Tsp}
    = \inner{\nabla_x u,\nabla_x v}_{L^2(\Omega)}
    = \inner{u,-\Delta_x v}_{L^2(\Omega)}$ for all $u\in\domain(\Bp)$ and $v\in C_0^\infty(\Omega)$.
This gives rise to the following strong, weak and ultra-weak operator
\begin{equation}\label{eq: strong, weak, ultra weak operator in space}
\begin{alignedat}{5}
    &-\Delta_x^{\mathrm{st}}&&: H^\Delta(\Omega) &&\ra L^2(\Omega),\quad && \inner{-\Delta_x^{\mathrm{st}}u,v}_{L^2(\Omega)} &&\deq \inner{-\Delta_x u,v}_{L^2(\Omega)},\\
    &-\Delta_x^{\mathrm{we}}&&: H^1(\Omega) &&\ra [H^1(\Omega)]',\quad && \pairingb{-\Delta_x^{\mathrm{we}}u,v}{H^1(\Omega)} &&\deq \inner{\nabla_x u,\nabla_x v}_{L^2(\Omega)},\\
    &-\Delta_x^{\mathrm{uw}}&&: L^2(\Omega) &&\ra [H^{\Delta}(\Omega)]',\quad && \pairingb{-\Delta_x^{\mathrm{uw}}u,v}{H^\Delta(\Omega)} &&\deq \inner{u,-\Delta_x v}_{L^2(\Omega)},
\end{alignedat}
\end{equation}
representing the starting point for a variational operator $\Bb^{\mathrm{st}}$, $\Bb^{\mathrm{we}}$ and $\Bb^{\mathrm{uw}}$, respectively. Now, we have to select a suitable domain $\domain(\Bb^{\mathrm{st}}),\domain(\Bb^{\mathrm{we}})$ and $\domain(\Bb^{\mathrm{uw}})$ as well as a suitable codomain $[\Ts^{\mathrm{st}}]',[\Ts^{\mathrm{we}}]'$ and $[\Ts^{\mathrm{uw}}]'$.
Choosing the largest possible test space $\Ts$ while still ensuring $\inner{\Bp u,v}_{\Tsp} = \pairing{\Bb u,v}{\Ts}$ for all $v\in\Ts$, we end up with $\Ts^{\mathrm{st}} \deq L^2(\Omega)$, $\Ts^{\mathrm{we}} \deq H^1_0(\Omega)$ and $\Ts^{\mathrm{uw}} \deq H^\Delta(\Omega)\cap H^1_0(\Omega)$, which are -- of course -- not the only possible options, but the most natural choices. Thereby, we note, that $\Ts^{\mathrm{st}},\Ts^{\mathrm{we}},\Ts^{\mathrm{uw}} \hookrightarrow_d L^2(\Omega)$. Now, regarding the domains, we first note, that $\domain(\Bp) \subseteq \domain(\Bb^{\mathrm{st}}),\domain(\Bb^{\mathrm{we}}),\domain(\Bb^{\mathrm{uw}})$ provides a lower bound. Second, regarding the boundary conditions, we note that in context of $\Bb^{\mathrm{uw}}$, the boundary condition is already implicitly imposed by the construction of the operator in form of a vanishing boundary integral, while for $\Bb^{\mathrm{st}}$ and $\Bb^{\mathrm{we}}$, the boundary condition needs to be imposed by the domain, i.e.,  $\domain(\Bb^{\mathrm{st}}),\domain(\Bb^{\mathrm{we}}) \subseteq H^1_0(\Omega)$. Hence, as an upper bound, we get $\domain(\Bb^{\mathrm{st}}) \subseteq \Asb^{\mathrm{st}} \deq H^\Delta(\Omega)\cap H^1_0(\Omega)$ as well as $\domain(\Bb^{\mathrm{we}})\subseteq \Asb^{\mathrm{we}} \deq H^1_0(\Omega)$ and $\domain(\Bb^{\mathrm{uw}}) \subseteq \Asb^{\mathrm{uw}} \deq L^2(\Omega)$. Thereby, we note, that each domain is a dense subspace of the Banach space $\Asb^{\mathrm{st}}$, $\Asb^{\mathrm{we}}$ and $\Asb^{\mathrm{uw}}$, respectively, as $\domain(\Bp)$ is already dense in each of them. In a similar fashion, we can derive several variational formulations for each example given in \S\ref{ch: introduction examples}, for which the following formulations will be considered in the scope of this paper.
\begin{remark}
    We stress, that $H^\Delta(\Omega) \cap H^1_0(\Omega) = H^2(\Omega)\cap H^1_0(\Omega)$ if $\Omega$ is quasi-convex \cite[Def. 8.9, Lem. 8.11]{GesztesyMitrea2011}, in particular if $\Omega$ is convex or has a smooth boundary. In this case $(\norm{\nabla\cdot}_{L^2(\Omega)}^2 + \norm{\Delta\cdot}_{L^2(\Omega)}^2)^{\frac{1}{2}}$ is a norm on $H^\Delta(\Omega)\cap H^1_0(\Omega)$ equivalent to the standard $H^2(\Omega)$ norm, e.g. \cite[Thm. 8.12]{GilbaT2001}.
\end{remark}

\begin{example}[Poisson equation]\label{ex: formulations for elliptic}
    For the Poisson equation in \cref{ex: elliptic problem}, set $\Tsp \deq L^2(\Omega)$, let $\domain(\Bp)$ as given in \cref{ex: classical domains} and consider the formulations
    \begin{compactenum}[(i)]
        \item\label{it:formulations for elliptic:strong} \emph{strong}: $\Ts\deq L^2(\Omega)$ and $\Asb \deq H^\Delta(\Omega)\cap H^1_0(\Omega)$ with $\Bb \deq -\Delta_x^{\mathrm{st}}\vert_{\domain(\Bb)} : \domain(\Bb) \subseteq\Asb\ra \Ts'$ for any domain $\domain(\Bb)$ such that $\domain(\Bp) \subseteq \domain(\Bb)\subseteq \Asb$;
        \item\label{it:formulations for elliptic:weak} \emph{weak}: $\Ts\deq H^1_0(\Omega)$ and $\Asb \deq H^1_0(\Omega)$ with $\Bb \deq -\Delta_x^{\mathrm{we}}\vert_{\domain(\Bb)} : \domain(\Bb) \subseteq\Asb\ra \Ts'$ for any domain $\domain(\Bb)$ such that $\domain(\Bp) \subseteq \domain(\Bb)\subseteq \Asb$;
        \item\label{it:formulations for elliptic:ultra-weak} \emph{ultra-weak}: $\Ts\deq H^\Delta(\Omega)\cap H^1_0(\Omega)$ and $\Asb \deq L^2(\Omega)$ with $\Bb \deq -\Delta_x^{\mathrm{uw}}\vert_{\domain(\Bb)} : \domain(\Bb) \subseteq\Asb\ra \Ts'$ for any domain $\domain(\Bb)$ such that $\domain(\Bp) \subseteq \domain(\Bb)\subseteq \Asb$.
    \end{compactenum}
    Thereby, we equip $H^\Delta(\Omega)\cap H^1_0(\Omega)$ with $(\norm{\nabla\cdot}_{L^2(\Omega)}^2 + \norm{\Delta\cdot}_{L^2(\Omega)}^2)^{\frac{1}{2}}$ as norm.
    Further holds in all three cases, that $\domain(\Bb) \subseteq_d\Asb$ as $\domain(\Bp)$ is already dense in $\Asb$. 
\end{example}

\begin{example}[Heat equation -- strong in time]\label{ex: strong formulation heat}
    For the heat equation given in \cref{ex: heat equation}, set $\Tsp \deq L^2(Q)$ and define the Hilbert spaces $\Ts \deq L^2(I;H^1_0(\Omega))$ and $\Asb \deq L^2(I;H^1_0(\Omega))\cap H^1_{0,}(I;H^{-1}(\Omega))$
    with their corresponding norms $\norm{\cdot}_{\Ts} \deq \norm{\nabla_x\cdot}_{L^2(Q)}$ and $\norm{\cdot}_{\Asb}^2 \deq \norm{\cdot}_{\Ts}^2 + \norm{\partial_t \cdot}_{\Ts'}^2$, respectively.
    We consider the variational operator $\Bb :\Asb \ra \Ts'$ arising from $\Bp$ in \cref{ex: heat equation} by one integration by parts in space, for all $u\in\Asb$ and $v\in\Ts$ given by $\pairing{\Bb u,v}{\Ts} \deq \inner{\partial_t u,v}_{L^2(Q)} +\Pairing{-\Delta_x^{\mathrm{we}} u,v}{L^2(I;H^{-1}(\Omega))}{L^2(I;H^1_0(\Omega))}$, see also \cite{DautrayLions,SchwabStevenson}.
\end{example}

\begin{example}[Heat equation -- weak in time]\label{ex: weak formulation heat}
    For the heat equation given in \cref{ex: heat equation}, set $\Tsp \deq L^2(Q)$ and define the Hilbert spaces $\Asb \deq L^2(I;H^1_0(\Omega))$ and $\Ts \deq L^2(I;H^1_0(\Omega))\cap H^1_{,0}(I;H^{-1}(\Omega))$
    with their corresponding norms $\norm{\cdot}_{\Asb} \deq \norm{\nabla_x\cdot}_{L^2(Q)}$ and $\norm{\cdot}_{\Ts}^2 \deq \norm{\cdot}_{\Asb}^2 + \norm{\partial_t \cdot}_{\Asb'}^2$, respectively.
    We consider the variational operator $\Bb :\Asb \ra \Ts'$ arising from $\Bp$ in \cref{ex: heat equation} by one integration by parts in time and one in space, for all $u\in\Asb$ and $v\in\Ts$ given by $\pairing{\Bb u,v}{\Ts} \deq -\inner{u,\partial_t v}_{L^2(Q)} +\Pairing{-\Delta_x^{\mathrm{we}} u,v}{L^2(I;H^{-1}(\Omega))}{L^2(I;H^1_0(\Omega))}$, see also \cite{BabuskaJanik,CheginiStevenson}.
\end{example}

\begin{example}[Wave equation -- strong in time]\label{ex: strong formulation wave}
    For the first order in time reformulation of the wave equation given in \cref{ex: FO wave equation}, set $\Tsp \deq L^2(Q;\R^2)$ and define the Hilbert spaces $\Asb \deq H^1_{0,}(I;L^2(\Omega)\times H^{-1}(\Omega)) \cap L^2(I;H^1_0(\Omega) \times L^2(\Omega))$ and $\Ts \deq L^2(I;L^2(\Omega)\times H^1_0(\Omega))$ with their corresponding norms $\norm{\cdot}_{\Ts} \deq \norm{\cdot}_{L^2(I;L^2(\Omega)\times H^1_0(\Omega))}$ and  $\norm{\cdot}_{\Asb}^2 \deq \norm{\partial_t\cdot}_{L^2(I;L^2(\Omega)\times H^{-1}(\Omega))}^2 + \norm{\cdot}_{L^2(I;H^1_0(\Omega) \times L^2(\Omega))}^2$,     respectively.
    Introducing the operator $A \deq \begin{psmallmatrix} 0 & -\Id\\ -\Delta_x^{\mathrm{we}} & 0 \end{psmallmatrix}: H^1_0(\Omega) \times L^2(\Omega) \ra L^2(\Omega)\times H^{-1}(\Omega)$,
    we consider the variational operator $\Bb :\Asb \ra \Ts'$ araising from $\Bp$ in \cref{ex: FO wave equation} by one integration by parts in space, for all $\vec{u}\in\Asb$ and $\vec{v}\in\Ts$ given by
    \begin{align*}
        \pairing{\Bb \vec{u},\vec{v}}{\Ts} &\deq \pairing{\partial_t \vec{u} + A \vec{u},\vec{v}}{\Ts}\\
        &= \inner{\partial_t u_1 - u_2,v_1}_{L^2(Q)}
        + \Pairing{\partial_t u_2-\Delta_x^{\mathrm{we}} u_1,v_2}{L^2(I;H^{-1}(\Omega))}{L^2(I;H^1_0(\Omega))}.
    \end{align*}
\end{example}

\begin{example}[Wave equation -- weak in time]\label{ex: weak formulation wave}
    For the wave equation given in \cref{ex: wave equation}, set $\Tsp \deq L^2(Q)$ and define the Hilbert spaces $\Asb \deq H^1_{0,}(I;H^1_0(\Omega))$ and $\Ts \deq H^1_{,0}(I; H^1_0(\Omega))$ with norms $\norm{\cdot}_{\Asb}^2 \equiv \norm{\cdot}_{\Ts}^2 \deq \norm{\partial_t \cdot}_{L^2(Q)}^2 + \norm{\nabla_x \cdot}_{L^2(Q)}^2$. 
    We consider the variational operator $\Bb :\Asb \ra \Ts'$ arising from $\Bp$ in \cref{ex: wave equation} by one integration by parts in time and one in space, for all $u\in\Asb$ and $v\in\Ts$ given by
    \begin{equation}\label{eq:applications:sym-wave:VF}
        \pairing{\Bb u,v}{\Ts} \deq -\inner{\partial_t u,\partial_t v}_{L^2(Q)} + \Pairing{-\Delta_x^{\mathrm{we}} u,v}{L^2(I;H^{-1}(\Omega))}{L^2(I;H^1_0(\Omega))}.
    \end{equation}
\end{example}

\begin{example}[Wave equation -- ultra-weak in time]\label{ex: ultra-weak formulation wave}
    For the first order in time reformulation of the wave equation given in \cref{ex: FO wave equation}, set $\Tsp \deq L^2(Q;\R^2)$ and define the Hilbert spaces $\Asb \deq L^2(I;H^1_0(\Omega)\times L^2(\Omega))$ and $\Ts \deq H^1_{,0}(I; H^{-1}(\Omega)\times L^2(\Omega)) \cap L^2(I;L^2(\Omega)\times H^1_0(\Omega))$ with their corresponding norms $\norm{\cdot}_{\Asb} \deq \norm{\cdot}_{L^2(I;H^1_0(\Omega)\times L^2(\Omega))}$ and $\norm{\cdot}_{\Ts}^2 \deq \norm{\partial_t\cdot}_{L^2(I;H^{-1}(\Omega)\times L^2(\Omega))}^2 + \norm{\cdot}_{L^2(I;L^2(\Omega)\times H^1_0(\Omega))}^2$, respectively. Using, that $-\Delta_x^{\mathrm{we}}$ and $\Id$ are self-adjoint, the adjoint of $A$ given in \cref{ex: strong formulation wave} reads $A^* = \begin{psmallmatrix} 0 & -\Delta_x^{\mathrm{we}}\\ -\Id & 0 \end{psmallmatrix}: L^2(\Omega) \times H^1_0(\Omega) \ra H^{-1}(\Omega) \times L^2(\Omega)$
    and we consider the variational operator $\Bb :\Asb \ra \Ts'$ araising from $\Bp$ in \cref{ex: FO wave equation} by one integration by parts in time (due to the first order reformulation, this effectively corresponds to two integration by parts in time, thus the name \emph{ultra-weak in time}) and one in space, for all $\vec{u}\in\Asb$ and $\vec{v}\in\Ts$ given by
    \begin{align*}
        \pairing{\Bb \vec{u},\vec{v}}{\Ts} &\deq \Pairing{\vec{u},-\partial_t\vec{v} + A^*\vec{v}}{\Asb}{\Asb'}\\
        &\kern-5pt= \Pairing{u_1,-\partial_t v_1 -\Delta_x^{\mathrm{we}} v_2}{L^2(I;H^1_0(\Omega))}{L^2(I;H^{-1}(\Omega))} + \inner{u_2,-\partial_t v_2 - v_1}_{L^2(Q)}.
    \end{align*}
\end{example}

\begin{remark}\label{rm: general variational spatial operators}
    The above variational formulations involving $-\Delta_x^{\mathrm{we}}$ can be extended directly to an arbitrary bounded and coercive (time variant) operator $A(t): H^1_0(\Omega) \ra H^{-1}(\Omega)$ instead of $-\Delta_x^{\mathrm{we}}$. In particular, for the elliptic operator $A_\circ$ defined in \cref{rm:general elliptic operaotrs possible} (omitting the time variance for brevity), we end up with the strong, weak and ultra weak operator
    \begin{align*}
        \inner{A^{\mathrm{st}} u, v}_{L^2(\Omega)} &\deq \inner{-\nabla_x \cdot\parens{\uline{A} \nabla_x u} + \uline{b} \cdot \nabla_x u + \uline{c} u, v}_{L^2(\Omega)},\\
        \pairingb{A^{\mathrm{we}} u, v}{H^1(\Omega)}  &\deq \inner{\uline{A} \nabla_x u,\nabla_x v}_{L^2(\Omega)} + \inner{\uline{b} \cdot \nabla_x u, v}_{L^2(\Omega)} + \inner{\uline{c} u,v}_{L^2(\Omega)},\\
        \pairingb{A^{\mathrm{uw}} u, v}{H^2(\Omega)}  &\deq \inner{u,-\nabla_x \cdot\parens{\uline{A}^\top \nabla_x v} -\uline{b} \cdot \nabla_x  v + \parens{\uline{c} -\nabla_x\cdot\uline{b}} v}_{L^2(\Omega)},
    \end{align*}
    respectively, replacing \eqref{eq: strong, weak, ultra weak operator in space} -- assuming that $\uline{A}$, $\uline{b}$ and $\uline{c}$ are sufficiently smooth.
\end{remark}
These examples already show that there is a huge variety regarding the choice of the test space $\Ts$ and the variational formulation \eqref{eq:operator-equation}, recalling that well-posedness is not required as \eqref{eq:operator-equation} only serves as a starting point.

\subsection{Completion and extension}\label{subsec: compl_ext}
Our goal is to \emph{construct a well-posed and optimally stable extension} $\B : \As \ra \Ts'$ of \eqref{eq:operator-equation}. Thereby, an extension of \eqref{eq:operator-equation} means that $\domain(\Bb) \subseteq_d \As$ and $\B\vert_{\domain(\Bb)} = \Bb$. To this end, it turns out that the following two properties are crucial:
\begin{enumerate}[label=(B\arabic*),font=\bfseries]
    \item\label{it: injective} $\Bb$ is injective on $\domain(\Bb)$;
    \item\label{it: range dense} the range $\range(\Bb)$ of $\Bb$ is dense in $\Ts'$ (weak surjectivity of $\Bb$).
\end{enumerate}
We first note a sufficient condition for \crefl{it: injective} and \crefl{it: range dense}, namely that \eqref{eq:operator-equation} is well-posed and stable on a dense subspace $\Tse \subseteq_d \Ts'$ of right-hand sides.
\begin{lemma}\label{lm: solution of equation sufficient for B1 B2}
    Let $\Tse \subseteq_d \Ts'$ be a normed subspace and $C <\infty$. If \eqref{eq:operator-equation} admits a solution $u^*\in\domain(\Bb)$ for each $f\in\Tse$ satisfying the stability estimate $\norm{u^*}_\Asb \leq C \norm{f}_\Tse$, then \crefl{it: injective} and \crefl{it: range dense} are valid.
\end{lemma}
\begin{proof}
    Consider the operator equation \eqref{eq:operator-equation} for $f=0$. Since $\Tse$ is a subspace, we have $0\in \Tse$, i.e., there exists some $u^*\in \dom(\Bb)$ such that $\Bb u^* = 0$ in $\Ts'$ and $\norm{u^*}_\Asb \leq C \norm{f}_\Tse = 0$ by assumption. Thus, if $\Bb u^* = 0$ then $u^* = 0$, i.e., the solution of the homogeneous problem is unique and hence $\kernel(\Bb) = \set{0}$, i.e., $\Bb$ is injective on $\domain(\Bb)$, i,e., \crefl{it: injective}.
    Since $\Tse$ is dense in $\Ts'$ and $\Tse \subseteq \range(\Bb) \subseteq \Ts'$, we conclude that $\range(\Bb)$ is dense in $\Ts'$, i.e., \crefl{it: range dense}.
\end{proof}
\begin{remark}\label{rm: solution of equation sufficient for B1 B2}
    Using that $(\Ts,\Tsp,\Ts')$ forms a Gelfand triple together with \cref{rm: continuous embedding}, possible choices for $\Tse$ include in particular $\Tse\subseteq_d \Tsp$ and $\Tse\subseteq_d \Ts$. In fact, this is the main reason, why the Gelfand triple $(\Ts,\Tsp,\Ts')$ is important as it is often sufficient to consider the operator equation for smooth right-hand sides, where existence results are typically easier to derive.
\end{remark}
If \crefl{it: injective} holds, then it is easy to see, that $\norm{\cdot}_{\domain(\Bb)} \deq \norm{\Bb\cdot}_{\Ts'}$ defines a norm on $\domain(\Bb)$. Thus, there exists a (up to isometric isomorphisms) uniquely determined completion $(\As,\norm{\cdot}_{\As})$ \label{inline: def X} of $(\domain(\Bb),\norm{\cdot}_{\domain(\Bb)})$, e.g. \cite[Theorem 2.32 \& Corollary 2.60]{EinsiW2017}.
Furthermore, there exists a uniquely determined continuous extension $\B\in\cL(\As,\Ts')$ of $\Bb$ from $\domain(\Bb)$ to $\As$, e.g.\ \cite[Proposition 2.59]{EinsiW2017}.
By definition, $\As$ is a Banach space, $\domain(\Bb) \subseteq_d \As$ and $\B\vert_{\domain(\Bb)} = \Bb$.
Furthermore, it is easy to see\footnote{By the representations of $\norm{\cdot}_{\As}$ and $\B$ given in the proofs of \cite[Theorem 2.32 \& Proposition 2.59]{EinsiW2017}, respectively, it holds for all $\bar{u}\in \As$, that $\norm{\bar{u}}_{\As} = \lim_{n\ra\infty} \norm{\Bb u_n}_{\Ts'} = \norm{\lim_{n\ra\infty} \Bb u_n}_{\Ts'} = \norm{\B \bar{u}}_{\Ts'}$, with $(u_n)_{n\in\N}\subset \domain(\Bb)$ denoting any sequence with $\lim_{n\ra\infty} u_n = \bar{u}$ in $\As$ (note that $\domain(\Bb) \subseteq_d \As$).}, that
\begin{equation}\label{eq: norm X = Bu}
    \norm{\bar{u}}_{\As} = \norm{\B \bar{u}}_{\Ts'},\, \bar{u}\in\As
    \quad\text{and}\quad
    \norm{u}_{\As}  = \norm{u}_{\domain(\Bb)} = \norm{\Bb u}_{\Ts'},\, u\in\domain(\Bb).
\end{equation}
Now, for $f\in\Ts'$, we seek the solution $\bar{u}^*\in \As$ of the extended operator equation
\begin{equation}\label{eq:well-posed operator-equation}
    \B \bar{u}^*=f\text{ in $\Ts'$},
    \quad\text{i.e.,}\quad
    \pairing{\B \bar{u}^*, v}{\Ts} = \pairing{f, v}{\Ts}\quad \forall v\in\Ts.
\end{equation}
By $\B\vert_{\domain(\Bb)} = \Bb$ and $\domain(\Bb)\subseteq \As$, it is easy to see, that every solution $u^*$ of \eqref{eq:operator-equation} also solves \eqref{eq:well-posed operator-equation}, but not the other way around. Hence, the extended problem \eqref{eq:well-posed operator-equation} is a weaker form of \eqref{eq:operator-equation}.
\begin{remark}\label{rm: notes on extended definition}
    \begin{compactenum}[(i)]
        \item If $\Bb$ is bounded (i.e.,  $\domain(\Bb) = \Asb$ and $\Bb\in\cL(\Asb,\Ts')$), it holds
        \begin{equation}\label{eq: embedding Asb in Asb if bounded}
            \Asb \hookrightarrow_d \As,\qquad\text{with}\qquad
            \norm{u}_{\As} \leq \norm{\Bb}_{\cL(\Asb,\Ts')} \norm{u}_{\Asb}\quad\forall u\in\Asb.
        \end{equation}
        \item The space $\As$ consists of the limits of all Cauchy sequences in $(\domain(\Bb),\norm{\cdot}_{\domain(\Bb)})$. The extension $\B$ reads $\B \bar{u} = \lim_{n\ra\infty} \Bb u_n$ for the limit  $\bar{u}\deq \lim_{n\to\infty}u_n$ of a Cauchy sequence $(u_n)_{n\in\N} \subset \domain(\Bb)$. Note, that $(u_n)_{n\in\N}$ is a Cauchy sequence in $\domain(\Bb)$ if and only if $(\Bb u_n)_{n\in\N}$ is a Cauchy sequence in $\Ts'$, i.e., $\B$ is well-defined, since $\Ts'$ is complete.
        \item The extended problem \eqref{eq:well-posed operator-equation} is independent of the chosen variational formulation \eqref{eq:operator-equation} for the classical problem \eqref{eq:analysis:classical problem}, as long as the test space $\Ts$ is fixed. In fact, the completion and the continuous extension are both unique. Thus, for any other operator $\Bb^\star : \domain(\Bb^\star) \subseteq \Asb^\star \ra \Ts'$ satisfying \crefl{it: injective} with $\domain(\Bb^\star) \subseteq_d \As$ and $\B \vert_{\domain(\Bb^\star)} = \Bb^\star$ we have $\overbar{U^\star} \equiv \As$ and $\overbar{B^\star} \equiv \B$, with $\overbar{U^\star}$ and $\overbar{B^\star}$ denoting the completion and continuous extension of $\domain(\Bb^\star)$ and $\Bb^\star$, respectively. See also \cref{rm: changing the elliptiv no impact} below.
    \end{compactenum}
\end{remark}
We now state our main result that the extended operator $\B$ is an isometric isomorphism, i.e., the operator equation \eqref{eq:well-posed operator-equation} is well-posed and optimally stable.
\begin{theorem}\label{tm:abstract setting well-posed}
    Let \crefl{it: injective,it: range dense} hold. The operator $\B$ from \eqref{eq:well-posed operator-equation} is
    \begin{compactenum}[(a)]
        \item an isomorphism, i.e., $\B\in \Lis(\As,\Ts')$ (well-posedness) and
        \item isometric, i.e., $\norm{\B}_{\cL(\As,\Ts')} = \norm{\B^{-1}}_{\cL(\Ts',\As)} = 1$
        (optimal stability).
    \end{compactenum}
\end{theorem}
\begin{proof}
    First, note that $\As$ and $\B$ are well-defined thanks to \crefl{it: injective}.
    By \eqref{eq: norm X = Bu} we have that the operator $\B: \As \to \range(\B)\subseteq \Ts'$ is an isometry, thus injective and $\norm{\B}_{\cL(\As,\Ts')} =1$. It remains to shown, that (i) $\B$ is surjective, i.e.,  $\range(\B)=\Ts'$, and (ii) $\norm{\B^{-1}}_{\cL(\Ts',\As)} = 1.$ To show the surjectivity, let $f\in\Ts'$ be arbitrary but fixed.
    By \crefl{it: range dense}, the range of $\Bb$ is dense in $\Ts'$. Hence, there exists a sequence $\sequence{u_n}_{n\in\N} \subset \domain(\Bb)$, such that $\lim_{n\ra\infty} \norm{f-\Bb u_n}_{\Ts'} = 0$.
    Thus, for all $n,m\in\N$
    \begin{equation*}
        \norm{u_n-u_m}_{\As}
        \overset{\eqref{eq: norm X = Bu}}{=}
        \norm{u_n-u_m}_{\domain(\Bb)}
        = \norm{\Bb(u_n-u_m)}_{\Ts'}
        \leq \norm{\Bb u_n -f }_{\Ts'}+\norm{\Bb u_m-f}_{\Ts'}
        \ \to\ 0
    \end{equation*}
    as $n,m\to\infty$, i.e., $\sequence{u_n}_{n\in\N}$ is a Cauchy sequence in $\As$.
    Hence, by the completeness of $\As$, there exists $\bar{u}\in \As$ with $\lim\limits_{n\ra\infty}\norm{\bar{u}-u_n}_{\As} = 0$, and it holds
    \begin{equation*}
        \norm{f-\B \bar{u}}_{\Ts'}
        \leq \norm{f - \Bb u_n}_{\Ts'}  + \norm{\Bb u_n - \B \bar{u}}_{\Ts'}
        =\norm{f - \Bb u_n}_{\Ts'} + \norm{u_n - \bar{u}}_{\As} \stackrel{n\to\infty}{\xrightarrow{\hspace{8mm}}}0,
    \end{equation*}
    i.e., $f=\B \bar{u} \in\range(\B)$ so that $\Ts'\subseteq\range(\B) \subseteq \Ts'$, where the second inclusion follows from \crefl{it: range dense}. Hence we conclude $\Ts' = \range(\B)$, i.e., (i). In order to show (ii), we have
    \begin{equation*}
        \norm{\B^{-1}}_{\cL(\Ts',\As)}
        = \sup_{f\in \Ts'\setminus\set{0}} \frac{\norm{\B^{-1}f}_{\As}}{\norm{f}_{\Ts'}}
        = \sup_{f\in \Ts'\setminus\set{0}} \frac{\norm{\B \B^{-1}f}_{\Ts'}}{\norm{f}_{\Ts'}}
        = \sup_{f\in \Ts'\setminus\set{0}} \frac{\norm{f}_{\Ts'}}{\norm{f}_{\Ts'}}
        = 1.
    \end{equation*}
    In particular, the inverse $\B^{-1}:\Ts'\ra\As$ is bounded (even isometric) and thus $\B \in \Lis(\As,\Ts')$, which concludes the proof.
\end{proof}
\begin{corollary}\label{co:abstract setting well-posed}
    Under the assumptions of \cref{tm:abstract setting well-posed}, we have, that
    \begin{compactenum}[(a)]
        \item the trial space $\As$ is reflexive as the test space $\Ts$ is reflexive;
        \item the trial space $\As$ is a Hilbert space w.r.t.\ the inner product $\inner{\cdot,\cdot}_{\As} \deq \inner{\B\cdot,\B\cdot}_{\Ts'}$ if and only if the test space $\Ts$ is a Hilbert space.
    \end{compactenum}
\end{corollary}
\begin{proof}
    \cref{tm:abstract setting well-posed} yields that $\As\cong\Ts'\cong\Ts$ are isomorphic.
\end{proof}
\begin{remark}\label{rm: connection bilinear and operator}
    \begin{compactenum}[(i)]
        \item The identity $\norm{\bar{u}^*}_{\As} = \norm{f}_{\Ts'}$ is of particular interest in numerical applications as it leads to the error-residual identity $\norm{\bar{u} - \bar{u}^*}_{\As} = \norm{\B \bar{u} - f}_{\Ts'}$ for all $\bar{u}\in\As$. This will be explored in Part II.
        \item The operator $\Bb$ and hence also $\B$ are often defined in terms of a bilinear ($\F=\R$) or sesquilinear ($\F=\C$) form $\bb :\domain(\Bb) \times\Ts \ra \F$ and $\b : \As \times \Ts \ra \F$, respectively.
        Then, \cref{tm:abstract setting well-posed} states, that $\b$ is continuous and inf-sup stable, with
       \begin{equation*}
            \left.\begin{array}{l}
                \overbar\gamma \mathclap{\phantom{\ds{\sup_{\bar{u}\in\As\setminus\set{0}}}}}  \\
                \overbar\beta   \mathclap{\phantom{\ds{\sup_{\bar{u}\in\As\setminus\set{0}}}}} \\
                \overbar\beta^* \mathclap{\phantom{\ds{\sup_{\bar{u}\in\As\setminus\set{0}}}}}
            \end{array}\right\}
            \deq
            \left\{\begin{array}{l}
                \adjustlimits\sup_{\bar{u}\in\As\setminus\set{0}} \sup_{v\in\Ts\setminus\set{0}} \\
                \adjustlimits\inf_{\bar{u}\in\As\setminus\set{0}} \sup_{v\in\Ts\setminus\set{0}} \\
                \adjustlimits\inf_{v\in\Ts\setminus\set{0}} \sup_{\bar{u}\in\As\setminus\set{0}}
            \end{array}\right\}
            \frac{\abs{\b(\bar{u},v)}}{\norm{\bar{u}}_{\As}\norm{v}_\Ts}
            = \left\{\begin{array}{l}
                \norm{\B}_{\cL(\As,\Ts')} \mathclap{\phantom{\ds{\sup_{\bar{u}\in\As\setminus\set{0}}}}}
                \\ \norm{\B^{-1}}_{\cL(\Ts',\As)}^{-1} \mathclap{\phantom{\ds{\sup_{\bar{u}\in\As\setminus\set{0}}}}}
                \\ \norm{\B^{-*}}_{\cL(\As',\Ts)}^{-1} \mathclap{\phantom{\ds{\sup_{\bar{u}\in\As\setminus\set{0}}}}}
            \end{array}\right\}
            = 1,
        \end{equation*}
        where $\B^{-*}:=(\B^*)^{-1}$ denotes the inverse of the adjoint of $\B$.
    \end{compactenum}
\end{remark}
Let us add a note regarding the norm $\norm{\cdot}_{\As} = \norm{\B\cdot}_{\Ts'}$ on $\As$. To this end, assume that $\Ts$ is a Hilbert space and denote the Riesz operator \eqref{eq: riesz operator} by $R_{\Ts}: \Ts \ra \Ts'$. Now, for $\overbar{u}\in \As$, defining $v_{\overbar{u}} \deq R_{\Ts}^{-1}\B \overbar{u} \in \Ts$, it holds $\norm{v_{\overbar{u}}}_{\Ts} = \norm{R_{\Ts}^{-1} \B \overbar{u}}_{\Ts} = \norm{\B \overbar{u}}_{\Ts'}$ and 
\begin{equation}\label{eq:analysis:abstract supremizer}
    \norm{\overbar{u}}_{\As}^2 = \norm{\B \overbar{u}}_{\Ts'}^2 = \inner{\B \overbar{u},\B \overbar{u}}_{\Ts'} = \pairing{\B \overbar{u},R_{\Ts}^{-1}\B \overbar{u}}{\Ts} = \pairing{\B \overbar{u},v_{\overbar{u}}}{\Ts}.
\end{equation}
Thus, $v_{\overbar{u}}$ is a supremizer of $\norm{\B \overbar{u}}_{\Ts'}$.
Before we continue, let us summarize the procedure how to derive an optimally stable well-posed formulation:    
\begin{blackbox}
    \begin{enumerate}[nosep,leftmargin=*]
        \item Starting from the classical formulation $\Bp : \domain(\Bp)\to C(\Omega)$, find a Gelfand triple $
        \Ts\hookrightarrow_d\Tsp \cong \Tsp' \hookrightarrow_d\Ts'$ for the test space and a variational form $\Bb u = f$ in $\Ts'$ for $\Bb: \domain(\Bb) \subseteq \Asb \ra \Ts'$, being a (possibly unbounded) linear operator defined on a linear subspace $\domain(\Bb)$ of a Banach space $\Asb$.
        \item Either show \crefl{it: injective} and \crefl{it: range dense} or find a normed subspace $\Tse \subseteq_d \Ts'$ such that $\Bb u = f$ is well-posed and stable on  $\Tse$ (i.e., only for $f\in\Tse$).    
        \item Build the unique completion $(\As,\norm{\cdot}_{\As})$ of $\domain(\Bb)$ w.r.t.\ $\norm{\cdot}_{\domain(\Bb)} \deq \norm{\Bb\cdot}_{\Ts'}$.
        \item Build the unique continuous extension $\B\in\cL(\As,\Ts')$ of $\Bb$ form $\domain(\Bb)$ to $\As$.
        \item[$\Ra$] $\B\in \Lis(\As,\Ts')$ (well-posedness)\\ and $\norm{\B}_{\cL(\As,\Ts')} = \norm{\B^{-1}}_{\cL(\Ts',\As)} = 1$ (optimal stability).
    \end{enumerate}
\end{blackbox}
%

\subsection{Gelfand triple for the trial space}\label{ch: gelfand trial}
The sufficient condition formulated in \cref{lm: solution of equation sufficient for B1 B2} and required for \cref{tm:abstract setting well-posed} relies on $(\Ts,\Tsp,\Ts')$, i.e., a Gelfand triple for the test space. Even more can be said, if also the trial space admits a Gelfand triple structure, which is in fact the case for all \cref{ex: formulations for elliptic,ex: strong formulation heat,ex: weak formulation heat,ex: strong formulation wave,ex: ultra-weak formulation wave,ex: weak formulation wave}, see \cref{ex: trial pivot space} below. To fix the notation, let $\Asp$ be a Hilbert space, such that $(\Asb,\Asp,\Asb')$ forms a Gelfand triple. 
It would be beneficial, if the Gelfand triple structure would carry over from $\Asb$ to $\As$. It turns out that this can be achieved by an  assumption similar to \cref{lm: solution of equation sufficient for B1 B2} ($\Bb$ is well-posed and stable on a dense subspace $\Tse \subseteq_d \Ts'$), but for the adjoint operator $\Bb^*$. In order to define $\Bb^*$, $\Bb$ needs to be densely defined, i.e., $\domain(\Bb)\subseteq_d \Asb$. Thus, we consider the following embeddings
\begin{equation}\label{eq: gelfand As}
    \domain(\Bb)\subseteq_d \Asb\hookrightarrow_d\Asp \cong \Asp' \hookrightarrow \Asb'.
\end{equation}
Then, following e.g.\ \cite[\S 2.6]{Brezi2011}, there exists a unique adjoint operator $\Bb^* : \domain(\Bb^*)\subseteq_d \Ts \ra \Asb'$ of the (possibly unbounded) operator $\Bb: \domain(\Bb)\subseteq_d \Asb \ra \Ts'$, with its domain
\begin{equation}\label{eq: adjoint operator domain definition}
    \domain(\Bb^*) \deq 
    \set[\Big]{v\in \Ts \ : \sup_{u\in\domain(\Bb)\setminus\set{0}} 
    \frac{\abs{\pairing{\Bb u,v}{\Ts}}}{\norm{u}_{\Asb}} < \infty}
\end{equation}
being dense in $\Ts$ (even $\domain(\Bb^*) = \Ts$ if $\Bb$ is bounded), and
\begin{equation}\label{eq: adjoint operator definition}
    \pairing{\Bb^*v,u}{\Asb} = \pairing{\Bb u,v}{\Ts}\qquad\forall u\in \domain(\Bb),\ \forall v\in \domain(\Bb^*).
\end{equation}
\begin{theorem}\label{tm: gelfand carrys}
    Let \crefl{it: injective}, \crefl{it: range dense} and \eqref{eq: gelfand As} hold, $\Ase \subseteq_d \Asp$ be a subset and $C^*<\infty$. If the adjoint problem of \eqref{eq:operator-equation}, namely to find $v^*\in \domain(\Bb^*)$ such that $\Bb^* v^* = g$  in $\Asb'$, admits a solution for each $g\in \Ase$ and the solution satisfies the stability estimate
    $\norm{v^*}_\Ts \leq C^* \norm{g}_\Asp$, then 
    \begin{equation}\label{eq: embedding As in Asp}
        \As \hookrightarrow_d \Asp \cong \Asp' \hookrightarrow_d \As'
        \qquad\text{with}\qquad
        \norm{\bar{u}}_\Asp\leq C^*\norm{\bar{u}}_{\As}\quad\forall \bar{u} \in \As.
    \end{equation}
\end{theorem}
\begin{proof}
    First, note that $\Bb^*$ is well-defined by $\domain(\Bb)\subseteq_d \Asb$, while $\As$ and $\B$ are well-defined by \crefl{it: injective}.
    We denote by $\mathcal{V} \deq \set{v \in \domain(\Bb^*)\, :\, \Bb^* v \in \Ase}\subseteq \domain(\Bb^*) \subseteq \Ts$ the set of all solutions of the adjoint problem with right-hand sides in $\Ase$. Thus, in particular $\norm{v}_\Ts \leq C^* \norm{\Bb^* v}_\Asp$ for all $v \in\mathcal{V}$ by the stability of the adjoint problem and $\Bb^* (\mathcal{V}) = \Ase$ as there exists a solution for all right-hand sides in $\Ase$. Thus, we get for all $u\in \domain(\Bb) \subseteq \Asp$, that
    \begin{align*}
        &\norm{u}_{\domain(\Bb)}
         \overset{\eqref{eq: norm X = Bu}}{=} \norm{\Bb u}_{\Ts'}
        = \sup_{v \in \Ts} \frac{\abs{\pairing{\Bb u,v}{\Ts}}}{\norm{v}_{\Ts}}
        \geq \sup_{v \in \mathcal{V}} \frac{\abs{\pairing{\Bb u,v}{\Ts}}}{\norm{v}_{\Ts}}
        \geq \sup_{v \in \mathcal{V}} \frac{\abs{\pairing{\Bb u,v}{\Ts}}}{C^* \norm{\Bb^* v}_\Asp}                                                        \\
         &\qquad \overset{\eqref{eq: adjoint operator definition}}{=} \sup_{v \in \mathcal{V}} \frac{\abs{\pairing{\Bb^* v, u}{\Asb}}}{C^* \norm{\Bb^* v}_\Asp}
        = \sup_{g\in \Ase} \frac{\abs{\pairing{g, u}{\Asb}}}{C^* \norm{g}_\Asp}
        \overset{\eqref{eq: gelfand pairing = product}}{=} \sup_{g\in \Ase} \frac{\abs{\inner{g, u}_{\Asp}}}{C^* \norm{g}_\Asp}
        \overset{(*)}{=} \sup_{g\in G} \frac{\abs{\inner{g, u}_{\Asp}}}{C^* \norm{g}_\Asp}                                                                    \\
         &\qquad \overset{\phantom{\eqref{eq: adjoint operator definition}}}{\geq} \frac{\abs{\inner{u, u}_{\Asp}}}{C^* \norm{u}_\Asp}
        = \frac{1}{C^*} \norm{u}_{\Asp},
    \end{align*}
    where we used in $(*)$, that $\Ase$ is dense in $\Asp$ (w.r.t.\ $\norm{\cdot}_{\Asp}$). Thus, the graph-type norm defined by $\norm{\cdot}_{\Bb}^2 \deq \norm{\cdot}_\Asp^2 + \norm{\Bb \cdot}_{\Ts'}^2$ of $\Bb$ is equivalent to $\norm{\cdot}_{\domain(\Bb)}$ on $\domain(\Bb)$ since for all $u\in\domain(\Bb)$
    \begin{equation*}
        \norm{u}_{\domain(\Bb)}^2
        \!\deq \norm{\Bb u}_{\Ts'}^2
        \leq \norm{u}_{\Bb}^2
        = \norm{u}_\Asp^2 + \norm{\Bb u}_{\Ts'}^2
        \leq (1+{C^*}^2) \norm{\Bb u}_{\Ts'}^2
        = (1+{C^*}^2) \norm{u}_{\domain(\Bb)}^2.
    \end{equation*}
    Denoting by $(\widetilde{\Asb},\norm{\cdot}_{\widetilde{\Asb}})$ the completion of $\domain(\Bb)$ w.r.t.\ $\norm{\cdot}_{\Bb}$ instead of $\norm{\cdot}_{\domain(\Bb)}$, we have $\widetilde{\Asb} \subseteq \Asp$ by the definition of $\norm{\cdot}_{\Bb}$ and the completeness of $\Asp$. Moreover, $\widetilde{\Asb} = \As$ (as sets) by the equivalence of $\norm{\cdot}_{\Bb}$ and $\norm{\cdot}_{\domain(\Bb)}$.
    Due to the definitions of $\norm{\cdot}_{\Bb}$ and $\norm{\cdot}_{\domain(\Bb)}$, we obtain that $\norm{\cdot}_{\widetilde{\Asb}}^2 = \norm{\cdot}_\Asp^2 + \norm{\B \cdot}_{\Ts'}^2$.
    Now, let $\bar{u}\in \As$ be arbitrary but fixed. By $\bar{u} \in \widetilde{\Asb}$ and $\domain(\Bb) \subseteq_d \widetilde{\Asb}$ there exists a sequence $(u_n)_{n\in\N} \subset \domain(\Bb)$ with $0 = \lim\limits_{n\ra\infty}\norm{\bar{u}-u_n}_{\widetilde{\Asb}}^2 = \lim\limits_{n\ra\infty} (\norm{\bar{u}-u_n}_\Asp^2 + \norm{\bar{u}-u_n}_{\As}^2)$, i.e., $\lim\limits_{n\ra\infty} u_n = \bar{u}$ w.r.t.\ $\norm{\cdot}_\Asp$ and $\norm{\cdot}_{\As}$. Thus,
    $
        \norm{\bar{u}}_{\As}
        = \lim_{n\ra\infty}\norm{u_n}_{\As}
        = \lim_{n\ra\infty} \norm{u_n}_{\domain(\Bb)}
        \geq \lim_{n\ra\infty} \frac{\norm{u_n}_{\Asp}}{C^*}
        = \frac{\norm{\bar{u}}_{\Asp}}{C^*}$,
    and we conclude $\As \hookrightarrow \Asp$.
    Finally, $\As \subseteq_d \Asp$ follows by $\domain(\Bb) \subseteq \As \subseteq \Asp$, and the fact that $\domain(\Bb)$ is dense in $\Asp$ by \eqref{eq: gelfand As} and \cref{rm: continuous embedding}. Finally, $\Asp'\hookrightarrow_d \As'$ holds by \cref{rm: gelfand} and the fact that $\As$ is reflexive by \cref{co:abstract setting well-posed}.
\end{proof}
\begin{remark}\label{rm: gelfand carrys}
    By \eqref{eq: gelfand As} and \cref{rm: continuous embedding}, possible choices for $\Ase$ include $\Ase\subseteq_d \Asp$, $\Ase\subseteq_d \Asb$ and $\Ase\subseteq_d \domain(\Bb)$, where the density in the latter is taken w.r.t.\ $\norm{\cdot}_{\Asb}$ or $\norm{\cdot}_{\Asp}$ but not $\norm{\cdot}_{\domain(\Bb)}$.
    Note that the stability estimate for the adjoint problem has to hold w.r.t.\ $\norm{\cdot}_\Asp$, unlike to \cref{lm: solution of equation sufficient for B1 B2}, where the norm could be chosen arbitrary.
\end{remark}
\begin{remark}\label{rm: notes on gelfand carrys}
    \begin{compactenum}[(i)]
        \item If $\Bb$ is bounded, it holds 
        $\Asb \hookrightarrow_d \As \hookrightarrow_d \Asp \cong \Asp' \hookrightarrow_d \As' \hookrightarrow_d \Asb'$, 
         using \eqref{eq: embedding Asb in Asb if bounded} and \eqref{eq: embedding As in Asp}.
        \item By \eqref{eq: embedding As in Asp}, the graph-type norm $\norm{\cdot}_{\B}^2 \deq \norm{\cdot}_\Asp^2 + \norm{\B \cdot}_{\Ts'}^2$ is equivalent to $\norm{\cdot}_{\As}$.
    \end{compactenum}
\end{remark}
Let us summarize the procedure.
\begin{blackbox}
    \begin{enumerate}[nosep,leftmargin=*]
        \item[5.] Check if the trial space admits a Gelfand structure $\Asb\hookrightarrow_d\Asp \cong \Asp' \hookrightarrow \Asb'$.
        \item[6.] Find a dense subset $\Ase \subseteq_d \Asp$ such that  the adjoint problem $\Bb^* v = g$ in $\Asb'$ is well-posed and stable on $\Ase$ (i.e. only for $g\in\Ase$).
        \item[$\Ra$] The extended trial space admits a Gelfand structure $\As\hookrightarrow_d\Asp \cong \Asp' \hookrightarrow_d \As'$.
    \end{enumerate}
\end{blackbox}

\begin{remark}
    Note that \cref{tm: gelfand carrys} implies \crefl{it: injective,it: range dense} but for $\Bb^*$ instead of $\Bb$ by a similar argument to that in \cref{lm: solution of equation sufficient for B1 B2}. These are precisely the assumption (A$^*$1) and (A$^*$2) in \cite{DahmeHSW2012}. There, (A$^*$1) and (A$^*$2) are used to construct a well-posed extension of $\Bb^*$, while we use these assumptions to add additional structure to an existing extension instead.
\end{remark}

\begin{example}\label{ex: trial pivot space}
    For the variational formulations in \cref{ex: formulations for elliptic,ex: strong formulation heat,ex: weak formulation heat,ex: strong formulation wave,ex: ultra-weak formulation wave,ex: weak formulation wave}, we have that $G\deq L^2(\Omega)$ for the Poisson equation, $G=L^2(Q)$ for the heat equation and the weak in time form of the wave equation as well as $G=L^2(Q,\R^2)$ for the strong and ultra-weak in time form of the wave equation.
\end{example}

%% file: sections/applications.tex
\section{Applications}\label{ch: applications}

Now, with the abstract operator framework at hand, let us consider the examples introduced in \S\ref{ch: introduction examples}, in particular, their variational formulations as given in \cref{ex: formulations for elliptic,ex: strong formulation heat,ex: weak formulation heat,ex: strong formulation wave,ex: ultra-weak formulation wave,ex: weak formulation wave} (together with the pivot space $\Asp$ as given in \cref{ex: trial pivot space}). For each of these formulations, we are now going to build their well-posed and optimally stable extension \eqref{eq:well-posed operator-equation} and characterize the extended trial space $\As$ and the extended operator $\B$ as far as possible.

\subsection{The Poisson equation}
\label{subsec:AppPoisson}
It is well known by Kellogg's theorem, e.g. \cite[Theorem 6.14]{GilbaT2001}, that -- under suitable regularity assumptions on the domain $\Omega$ -- the classical formulation of the Poisson problem given in \cref{ex: elliptic problem} has a unique solution\footnote{$C^{k,\alpha}(\overline\Omega)$, $k\in\N$, is the Hölder space with Hölder exponent $\alpha\!\in\! (0,1]$; $C^{k,\alpha}_0(\overline\Omega) = C^{k,\alpha}(\overline\Omega) \cap C_0(\overline\Omega)$.} $u^* \in C_0^{2,\alpha}(\overline{\Omega})$ satisfying $\norm{u^*}_{C^{2,\alpha}(\overline{\Omega})} \leq C \norm{f}_{C^{0,\alpha}(\overline{\Omega})}$ for all right-hand sides $f\in C^{0,\alpha}(\overline{\Omega})$ and some $C<\infty$.
As it holds $C^\infty_0(\Omega) \subset C^{k,\alpha}_0(\overline{\Omega}) \subset C^k_0(\Omega) \subset L^2(\Omega)$, we conclude, that $u^* \in \domain(\Bp)$ and, by the density of $C^\infty_0(\Omega)$ in $L^2(\Omega)$, that $C^{0,\alpha}(\overline{\Omega})$ is dense in $L^2(\Omega)$. Thus, $u^*$ is a solution for each variational formulation given in \cref{ex: formulations for elliptic} and for each formulation holds \cref{tm:abstract setting well-posed} by \cref{lm: solution of equation sufficient for B1 B2,rm: solution of equation sufficient for B1 B2}. We will now see, that the abstract framework introduced in \S\ref{ch: theory} recovers the well-known formulations for the Poisson problem. 

\subsubsection{Strong formulation}
\label{ch:applications:strong Poisson}
Consider the operator $\Bb \deq -\Delta_x^{\mathrm{st}}\vert_{\domain(\Bb)} : \domain(\Bb) \subseteq_d H^\Delta(\Omega)\cap H^1_0(\Omega) \ra L^2(\Omega)$ given in \cref{ex: formulations for elliptic} (\ref{it:formulations for elliptic:strong}). Then, writing $-\Delta \equiv -\Delta_x^{\mathrm{st}}$, it holds for all $u\in \domain(\Bb)$, that
$\norm{\Bb u}_{\Ts'} = \norm{-\Delta_x^{\mathrm{st}}\vert_{\domain(\Bb)} u}_{L^2(\Omega)} = \norm{\Delta u}_{L^2(\Omega)}$,
while the norm on $\Asb \deq H^\Delta(\Omega)\cap H^1_0(\Omega)$ reads $\norm{\cdot}_{\Asb}^2 = \norm{\nabla \cdot}_{L^2(\Omega)}^2 + \norm{\Delta \cdot}_{L^2(\Omega)}^2$. Note, that the norm induced by $\Bb$ and the norm on $\Asb$ are equivalent, as it holds by integration by parts, the Cauchy–Schwarz inequality and Poincaré's inequality, that
\begin{align*}
    \norm{\nabla u}_{L^2(\Omega)}^2 
    &= \inner{\nabla u,\nabla u}_{L^2(\Omega)}
    = \inner{u,-\Delta u}_{L^2(\Omega)} + \underbrace{\inner{u,\vec n\cdot \nabla u}_{L^2(\partial\Omega)}}_{=0\text{ as } u=0 \text{ on } \partial\Omega}\\
    &\leq \norm{u}_{L^2(\Omega)} \norm{\Delta u}_{L^2(\Omega)}
    \leq C_{\Omega}\norm{\nabla u}_{L^2(\Omega)} \norm{\Delta u}_{L^2(\Omega)},
\end{align*}
for all $u \in H^\Delta(\Omega)\cap H^1_0(\Omega)$, where $C_\Omega < \infty$ denotes the Poincaré constant on $\Omega$. Dividing by $\norm{\nabla u}_{L^2(\Omega)}$ then gives the norm equivalence as it holds for all $u \in H^\Delta(\Omega)\cap H^1_0(\Omega)$, that
\begin{equation}\label{eq:applications:strong elliptic norm equivalence}
    \norm{\Delta u}_{L^2(\Omega)}^2 \leq \norm{\nabla u}_{L^2(\Omega)}^2 + \norm{\Delta u}_{L^2(\Omega)}^2
    \leq (1+C_\Omega^2)\norm{\Delta u}_{L^2(\Omega)}^2.
\end{equation}
As $\domain(\Bb) \subseteq_d H^\Delta(\Omega)\cap H^1_0(\Omega)$ and $H^\Delta(\Omega)\cap H^1_0(\Omega)$ is complete with respect to $\norm{\nabla \cdot}_{L^2(\Omega)} + \norm{\Delta\cdot}_{L^2(\Omega)}$, it holds for the well-posed and optimally stable completion \eqref{eq:well-posed operator-equation}, that $\As = H^\Delta(\Omega)\cap H^1_0(\Omega)$ and $\B = -\Delta_x^{\mathrm{st}}$.

\subsubsection{Weak formulation} 
\label{ch:applications:weak Poisson}
Consider the operator $\Bb \deq -\Delta_x^{\mathrm{we}}\vert_{\domain(\Bb)} : \domain(\Bb) \subseteq_d H^1_0(\Omega) \ra H^{-1}(\Omega)$ given in \cref{ex: formulations for elliptic} (\ref{it:formulations for elliptic:weak}). As the isometric Riesz operator \eqref{eq: riesz operator} on $H^1_0(\Omega)$ reads $R_{H^1_0(\Omega)} \equiv -\Delta_x^{\mathrm{we}}$, it holds for all $u\in \domain(\Bb)$, that
\begin{equation*}
    \norm{\Bb u}_{\Ts'} = \norm{-\Delta_x^{\mathrm{we}}\vert_{\domain(\Bb)} u}_{H^{-1}(\Omega)} = \norm{-\Delta_x^{\mathrm{we}} u}_{H^{-1}(\Omega)} = \norm{u}_{H^1_0(\Omega)} = \norm{\nabla u}_{L^2(\Omega)}.
\end{equation*}
Hence, the norm induced by the operator $\Bb$ is the $H^1$-seminorm. As $\domain(\Bb) \subseteq_d H^1_0(\Omega)$ and $H^1_0(\Omega)$ is complete with respect to $\norm{\nabla \cdot}_{L^2(\Omega)}$, it holds for the well-posed and optimally stable completion \eqref{eq:well-posed operator-equation}, that $\As = H^1_0(\Omega)$ and $\B = -\Delta_x^{\mathrm{we}}$.

\subsubsection{Ultra-weak formulation}
\label{ch:applications:utraweak Poisson}
Consider the operator $\Bb \deq -\Delta_x^{\mathrm{uw}}\vert_{\domain(\Bb)} : \domain(\Bb) \subseteq_d L^2(\Omega) \ra [H^\Delta(\Omega)\cap H^1_0(\Omega)]'$ given in \cref{ex: formulations for elliptic} (\ref{it:formulations for elliptic:ultra-weak}). Let $u \in \domain(\Bb)$ be arbitrary. We have already seen in context of the strong formulation, that $-\Delta\equiv -\Delta^{\mathrm{st}}_x : H^\Delta(\Omega)\cap H^1_0(\Omega) \ra L^2(\Omega)$ is an isomorphism and it holds \eqref{eq:applications:strong elliptic norm equivalence}, i.e., there exists a $v_u \in H^\Delta(\Omega)\cap H^1_0(\Omega)$ such that $-\Delta v_u = u$ in $L^2(\Omega)$ and $\norm{u}_{L^2(\Omega)} \leq \norm{v_u}_{\Ts} \leq \sqrt{(1+C_\Omega^2)}\norm{u}_{L^2(\Omega)}$ with $\norm{\cdot}_{\Ts}^2 \deq \norm{\nabla \cdot}_{L^2(\Omega)}^2 + \norm{\Delta \cdot}_{L^2(\Omega)}^2$.
Thus, it holds for $v_u$, that
\begin{equation*}
    \frac{\abs{\pairing{\Bb u,v_u}{\Ts}}}{\norm{v_u}_{\Ts}}
    = \frac{\abs{\inner{u,\Delta v_u}_{L^2(\Omega)}}}{\norm{v_u}_{\Ts}}
    = \frac{\abs{\inner{u,u}_{L^2(\Omega)}}}{\norm{v_u}_{\Ts}}
    \geq \frac{1}{\sqrt{(1+C_\Omega^2)}}\norm{u}_{L^2(\Omega)},
\end{equation*}
and for all $v\in H^\Delta(\Omega)\cap H^1_0(\Omega)$, that
\begin{equation*}
    \frac{\abs{\pairing{\Bb u,v}{\Ts}}}{\norm{v}_{\Ts}}
    = \frac{\abs{\inner{u,\Delta v}_{L^2(\Omega)}}}{\norm{v}_{\Ts}}
    \leq \frac{\norm{u}_{L^2(\Omega)}\norm{\Delta v}_{L^2(\Omega)}}{\norm{v}_{\Ts}}
    \leq \frac{\norm{u}_{L^2(\Omega)}\norm{v}_{\Ts}}{\norm{v}_{\Ts}}
    = \norm{u}_{L^2(\Omega)}.
\end{equation*}
Hence, we conclude, that the norm induced by the operator and the $L^2$-norm are equivalent as it holds $\frac{1}{\sqrt{(1+C_\Omega^2)}}\norm{u}_{L^2(\Omega)} \leq \norm{\Bb u}_{[H^\Delta(\Omega)\cap H^1_0(\Omega)]'} \leq \norm{u}_{L^2(\Omega)}$.
As $\domain(\Bb) \subseteq_d L^2(\Omega)$ and $L^2(\Omega)$ is complete with respect to $\norm{\cdot}_{L^2(\Omega)}$, it holds for the well-posed and optimally stable completion \eqref{eq:well-posed operator-equation}, that $\As = L^2(\Omega)$ and $\B = -\Delta_x^{\mathrm{uw}}$.

\begin{remark}\label{rm: changing the elliptiv no impact}
    In the three settings given in \cref{ex: formulations for elliptic}, the domain $\domain(\Bb)$ was not fixed and was allowd to range from $\domain(\Bp)$ to $\Asb$. Choosing $\domain(\Bb) = \domain(\Bp)$ we can replace $\Bb$ by $\Bp$ as it holds $\Bb\vert_{\domain(\Bp)} = \Bp$ in $\Ts'$ by the construction of $\Bb$ and $\Ts'$. Further, we could also replace $\Asb$ by\footnote{Note that the upper bound of $\domain(\Bb)$ is fixed and does not increase with $\Asb$.} $L^2(\Omega)$. Even after applying these changes (regarding $\Bb$, $\domain(\Bb)$ and $\Asb$) to each of the three formulations -- in particular, we can choose $\Bb = \Bp$, $\domain(\Bb)=\domain(\Bp)$ and $\Asb = L^2(\Omega)$ for all three cases -- they still lead to the  same well-posed extensions $\As$ and $\B$ as presented above.
    Thus, $\Ts$ is the sole quantity that encodes the regularity of the resulting extension, i.e., the choice of $\Ts$ determines if we end up with $-\Delta_x^{\mathrm{st}}$, $-\Delta_x^{\mathrm{we}}$ or $-\Delta_x^{\mathrm{uw}}$.
\end{remark}
\begin{remark}\label{rm: elliptic result hold for general operator}
    The above results still hold for general bounded elliptic second-order differential operators as defined in $\cref{rm: general variational spatial operators}$ although the constants of the norm equivalence between $\norm{\cdot}_{\Asb}$ and $\norm{\Bb\cdot}_{\Ts'}$ will change. 
    In case of the weak formulation, this coincides with the well-known result that $\norm{\Bb\cdot}_{H^{-1}(\Omega)}$  defines an equivalent norm to $\norm{\cdot}_{H^1_0(\Omega)}$ on $H^1_0(\Omega)$ for any bounded self-adjoint and coercive operator\footnote{By $R_{H^1_0(\Omega)} \equiv \Bb$ holds $\norm{\Bb\cdot}_{H^{-1}(\Omega)}^2 = \inner{\Bb\cdot,\Bb\cdot}_{H^{-1}(\Omega)} = \Pairing{\Bb\cdot,\cdot}{H^{-1}(\Omega)}{H^1_0(\Omega)}.$ Then, the norm equivalence immediately follows by the boundedness and coercivity of $\Bb$, while the self-adjointness gives rise not only to a norm but also to an inner product.} $\Bb: H^1_0(\Omega)\ra H^{-1}(\Omega)$. In this case, the isometric Riesz operator \eqref{eq: riesz operator}  on $H^1_0(\Omega)$ equipped with the norm $\norm{\Bb\cdot}_{H^{-1}(\Omega)}$ reads $R_{H^1_0(\Omega)} \equiv \Bb$ instead of $R_{H^1_0(\Omega)} \equiv -\Delta_x^{\mathrm{we}}$.
\end{remark}

\subsection{The heat equation}\label{subsec:AppHeat}
Let us consider the heat equation in \cref{ex: heat equation}.
\subsubsection{Strong in time}
\label{ch:applications:strong in time heat}
Starting with the strong in time variational formulation given in \cref{ex: strong formulation heat}.
It is well known, that $\Bb u=f$ in $\Ts'$ admits a unique solution $u^*\in\Asb$ for each $f\in \Ts'$, \cite{FuhrerK2021,SchwabStevenson,SteinZ2020}. In particular, $\Bb \in \cL(\Asb,\Ts)$ and there holds the inf-sup stability 
\begin{equation}\label{eq:Applications:heat-inf-sup}
    \norm{u}_\Asb \leq \sup_{0\neq v\in \Ts}\frac{\pairing{\Bb u,v}{\Ts}}{\norm{v}_\Ts} = \norm{\Bb u}_{\Ts'},
\end{equation}
for all $u\in \Asb$. Thus, we have the norm equivalence 
\begin{equation}\label{eq:Applications:heat-norm-equivalence}
    \norm{u}_\Asb \leq \norm{\Bb u}_{\Ts'}\leq \sqrt{2}\norm{u}_\Asb\quad \forall u\in \Asb,
\end{equation}
and we immediately see, that the norm induced by the operator is equivalent to the norm on $\Asb$. Thus, applying our abstract framework to the heat equation does not do much, in fact, as $\norm{\cdot}_{\Asb}$ and $\norm{\Bb\cdot}_{\Ts'}$ are equivalent on $\Asb$, and $\Asb$ is already complete with respect to $\norm{\cdot}_{\Asb}$, the completion of $\Asb$ with respect to $\norm{\Bb\cdot}_{\Ts'}$ simply reads $\As \equiv \Asb$ and $\norm{\cdot}_{\As} = \norm{\Bb\cdot}_{\Ts'}$ as it holds $\B \equiv \Bb$. Hence, by going from $\Asb$ to $\As$, we effectively replaced the norm by an equivalent one. Thus, let us use our abstract framework to compute this norm $\norm{\Bb\cdot}_{\Ts'}$. Therefore, we need the Riesz operator \eqref{eq: riesz operator} on $\Ts$ given by $R_{\Ts} \equiv -\Delta_x^{\mathrm{we}}$, i.e., $\Bb = \partial_t + R_\Ts$.
Then using \eqref{eq:analysis:abstract supremizer}, we compute for $u\in\Asb$, that
\begin{equation}\label{eq:Applications:heat-new-norm-representation}
    \begin{aligned}
        \norm{\Bb u}_{\Ts'}^2 &= \pairing{\Bb u,R_\Ts^{-1}\Bb u}{\Ts} = \pairing{ \partial_tu + R_\Ts u, R_\Ts^{-1} (\partial_t u+R_\Ts u)}{\Ts}\\
        & = \pairing{\partial_t u,R_\Ts^{-1}\partial_t u}{\Ts} + 2 \pairing{\partial_t u,u}{\Ts} + \pairing{R_\Ts u,u}{\Ts}\\
        & = \norm{\partial_t u}_{\Ts'}^2 + \norm{u}_{\Ts}^2 + 2\pairing{\partial_t u,u}{\Ts} = \norm{u}_{\Asb}^2+ 2\pairing{\partial_t u,u}{\Ts}. 
    \end{aligned}
\end{equation}
The norm equivalence \eqref{eq:Applications:heat-norm-equivalence} thus relies on the fact that 
\begin{align*}
    0&\leq \norm{u(T)}_{L^2(\Omega)}^2 = \bintegral{\Omega}{}{u(T,x)^2}{x} \overset{u(0)=0}= \bintegral{\Omega}{}{\bintegral{I}{}{\partial_tu(t,x)^2}{t}}{x} = 2\pairing{\partial_t u,u}{\Ts}\\
    &\leq 2\norm{\partial_t u}_{\Ts'}\norm{u}_{\Ts}\leq \norm{\partial_t u}_{\Ts'}^2+\norm{u}_{\Ts}^2 = \norm{u}_{\Asb}^2.
\end{align*}
Hence, we get $\B = \Bb = \partial_t - \Delta_x^{\mathrm{we}}$ and $\As = \Asb = L^2(I;H^1_0(\Omega))\cap H^1_{0,}(I;H^{-1}(\Omega))$ with norm $\norm{\cdot}_{\As} = \sqrt{\norm{\cdot}_{\Asb}^2 + 2\pairing{\partial_t \cdot,\cdot}{\Ts}}$.
\begin{remark}
    As seen above, the mixed term in the norm can be written as $2\pairing{\partial_t u,u}{\Ts} = \norm{u(T)}_{L^2(\Omega)}^2$,
    i.e., the norm $\norm{\cdot}_{\As}$ and therefore the well-posed formulation \eqref{eq:well-posed operator-equation} of the heat equation corresponds exactly to the well-posed and optimally stable formulation introduced in \cite{UrbanP2013}.
\end{remark}
\begin{remark}
    Note, that the supremizer of \eqref{eq:Applications:heat-inf-sup} is given in \eqref{eq:analysis:abstract supremizer} as $v_u = R_\Ts^{-1}\Bb u = R_\Ts^{-1}\partial_t u + u$, which was already used in \cite{ErnGuermond,SchwabStevenson,Steinbach2015,TantardiniVeeser2016,SpaceTimeUrbanPatera1,UrbanP2013} to show the stability of the formulation. 
\end{remark}
\begin{remark}\label{rm: heat result hold for general operator}
    The above results hold true when replacing $-\Delta_x^{\textrm{we}}$ by an arbitrary bounded self-adjoint coercive operator $A_x:H^1_0(\Omega) \ra H^{-1}(\Omega)$ as we can replace the norm on $H^1_0(\Omega)$ by the equivalent norm $\norm{A_x\cdot}_{H^{-1}(\Omega)}$ as mentioned in \cref{rm: elliptic result hold for general operator}.
\end{remark}

\subsubsection{Weak in time}
\label{ch:applications:weak in time heat}
Now, let us consider the weak in time variational formulation given in \cref{ex: weak formulation heat}, which is exactly the formulation used e.g. in \cite{BabuskaJanik,CheginiStevenson}. It was shown in \cite[Thm. 2.2]{CheginiStevenson}, that $\Bb \in \Lis(\Asb,\Ts')$, i.e. \eqref{eq:operator-equation} is already well-posed and we immediately get $\As \equiv \Asb$ with $\norm{\cdot}_{\As}$ being equivalent to $\norm{\cdot}_{\Asb}$, and $\B \equiv \Bb$. Hence, the application of our framework reduces to changing the norm on the trial space and thus adding optimal stability to the already given well-posedness. Unlike to the strong in time formulation above, there is no explicit representation for the Riesz operator $R_{\Ts}$. Hence, we can not give an representation for $\norm{\cdot}_{\As}$ by a similar calculation to \eqref{eq:Applications:heat-new-norm-representation}.

\subsection{The wave equation}\label{ch: wave equation}
Next, we consider the wave equation in \cref{ex: wave equation}.

\subsubsection{Strong in time}
\label{ch:applications:strong in time wave}
We start with the strong in time variational formulation of the wave equation as given in \cref{ex: strong formulation wave}. We denote by $\mathbb{F}$ the 2D flip operator, i.e. for two Banach spaces $\Ase_1$ and $\Ase_2$, their vector product $\Ase \deq \Ase_1\times\Ase_2$ and $\begin{psmallmatrix}x_1\\x_2\end{psmallmatrix}  \in \Ase$, we have
\begin{gather*}
    \mathbb{F}\begin{psmallmatrix}x_1\\x_2\end{psmallmatrix} = \begin{psmallmatrix}x_2\\x_1\end{psmallmatrix},\qquad
    \mathbb{F}(\Ase_1\times\Ase_2) = \Ase_2\times\Ase_1,\qquad
    \norm{\cdot}_{\mathbb{F}\Ase} = \norm{\mathbb{F}\cdot}_{\Ase},\\
    (\mathbb{F} \Ase)' = \mathbb{F}\Ase',\qquad
    \mathbb{F}L^2(I;\Ase) = L^2(I;\mathbb{F}\Ase).
\end{gather*}
Thus, with $W \deq L^2(\Omega) \times H^1_0(\Omega)$, we get $\Asb = H^1_{0,}(I;W') \cap L^2(I;\mathbb{F}W)$ and $\Ts = L^2(I;W)$ and their norms are given by $\norm{\vec{w}}_{W}^2 \deq \norm{w_1}_{L^2(\Omega)}^2 + \norm{w_2}_{H^1_0(\Omega)}^2$, $\norm{\cdot}_{\Asb}^2 \deq \norm{\partial_t\cdot}_{\Ts'}^2 + \norm{\cdot}_{\mathbb{F}\Ts}^2$ and $\norm{\cdot}_{\Ts} \deq \norm{\cdot}_{L^2(I;W)}$. 
Now, for $\vec{f} \in\Ts'$, the variational formulation in \cref{ex: strong formulation wave} amounts to find $\vec{u}^* \in \Asb$ such that
\begin{equation}\label{eq:Applications:wave-system-first-order-VF}
    \Bb\vec{u}^* = \vec{f}\text{ in }\Ts'\qquad\text{with}\qquad
    \pairing{\Bb \vec{u},\vec{v}}{\Ts} \deq \pairing{\partial_t \vec{u}+A \vec{u},\vec{v}}{\Ts},    
\end{equation}
for all $\vec{u}\in\Asb$ and all $\vec{v}\in\Ts$ and $A \deq \begin{psmallmatrix} 0 & -\Id\\ -\Delta_x^{\mathrm{we}} & 0 \end{psmallmatrix}: \mathbb{F} W \ra W'$. 
Further, by $W' = L^2(\Omega)\times H^{-1}(\Omega)$, the dual space and the dual norm of $\Ts'$ read
\begin{equation}\label{eq:Applications:strong wave:dual Ts}
    \Ts' = L^2(I;W'),\qquad
    \norm{\cdot}_{\Ts'}^2 = \norm{\cdot}_{L^2(Q)}^2 + \norm{\cdot}_{L^2(I;H^{-1}(\Omega))}^2.
\end{equation}
\begin{remark}\label{lm:application:wave operator bounded}
    The operator $\Bb$ in \eqref{eq:Applications:wave-system-first-order-VF} is bounded, i.e. $\Bb \in \cL(\Asb,\Ts')$, as it holds for $\vec{u}\in\Asb$, using $R_{H^1_0(\Omega)} \equiv -\Delta_x^{\mathrm{we}}$ for the Riesz operator \eqref{eq: riesz operator} on $H^1_0(\Omega)$, that
    \begin{align*}
        \norm{\Bb\vec{u}}_{\Ts'}^2
        &\overset{\eqref{eq:Applications:strong wave:dual Ts}}{=} \norm{\partial_t u_1-u_2}_{L^2(Q)}^2 + \norm{(\Delta_x^{\mathrm{we}})^{-1} \parens{\partial_t u_2 - \Delta_x^{\mathrm{we}}u_1}}_{L^2(I;H^1_0(\Omega))}^2\\
        &\leq 2\parens[\big]{\norm{\partial_t u_1}_{L^2(Q)}^2 + \norm{u_2}_{L^2(Q)}^2 + \norm{\partial_t u_2}_{L^2(I;H^{-1}(\Omega))}^2 +\norm{u_1}_{L^2(I;H^1_0(\Omega))}^2}\\
        &= 2\parens{\norm{\vec{u}}_{\mathbb{F}\Ts}^2 + \norm{\partial_t\vec{u}}_{\Ts'}^2}
        = 2 \norm{\vec{u}}_{\Asb}^2.
    \end{align*}
\end{remark}
\begin{theorem}[{\cite[Ch.\,3 Thm.\,8.1 \& eq.\,(8.15)]{LionsM1972}}]\label{thm:Applications:strong-wave-solution}
    Let $f\in L^2(Q)$, then there exists a unique $u^*\in H^2(I;H^{-1}(\Omega)) \cap H^1(I;L^2(\Omega))\cap L^2(I;H^1_0(\Omega))$ such that 
    \begin{equation*}
        \partial_{tt} u^* - \Delta^{\mathrm{we}}_x u^* = f \text{ in } L^2(I;H^{-1}(\Omega)),\quad u^*(0) = 0,\quad \partial_t u^*(0) = 0. 
    \end{equation*}
    Further, there exists $\hat C<\infty$ independent of $u^*$ and $f$ such that $\norm{\partial_t u^*}_{L^2(Q)}^2 + \norm{u^*}_{L^2(I;H^1_0(\Omega))}^2 \leq \hat C \norm{f}_{L^2(Q)}^2$.
\end{theorem}
\begin{remark}\label{rm:Applications:strong-wave-solution:constant}
    Using the techniques of \cite[Theorem 5.1, Remark 4.6]{SteinZ2020}, we can show that the constant in the preceding theorem equals $\hat C=\frac{T^2}{2}$ and the estimate is in fact sharp in the powers of $T$. 
\end{remark}
\begin{corollary}\label{co:Applications:wave-system-B1-B2}
    The variational formulation \eqref{eq:Applications:wave-system-first-order-VF} possesses a solution $\vec{u}^*\in\Asb$ for each $\vec{f} \in \Tse \deq H^1_{0,}(I;L^2(\Omega))\times L^2(Q)$. Furthermore, there exists $C_Q<\infty$ independent of $\vec{u}^*$ and $\vec{f}$ such that $\norm{\vec{u}^*}_{\Asb} \leq C_Q \norm{\vec{f}}_{L^2(Q;\R^2)}$. In particular \cref{tm:abstract setting well-posed} holds by \cref{lm: solution of equation sufficient for B1 B2,rm: solution of equation sufficient for B1 B2}.
\end{corollary}
\begin{proof}
    For $\vec{f} = (f_1,f_2) \in \Tse$ define $f \deq \partial_t f_1 + f_2 \in L^2(Q)$ and denote by $u^*$ and $C$ the solution and the constant given by \cref{thm:Applications:strong-wave-solution} with respect to the right-hand side $f$. Now, defining $\vec{u}^* \equiv (u^*_1,u^*_2) \deq (u^*,\partial_t u^* - f_1)$, it holds $\vec{u}^* \in\Asb$,  $\partial_t u^*_1 - u^*_2 = f_1$ in $L^2(Q)$ and $\partial_t u^*_2 - \Delta_x^{\mathrm{we}} u^*_1 = \partial_{tt}u^* - \partial_t f_1 - \Delta_x^{\mathrm{we}} u^* = f_2$ in $L^2(I;H^{-1}(\Omega))$.
    Hence, we conclude, that $\vec{u}^*$ solves \eqref{eq:Applications:wave-system-first-order-VF} as it holds for all $\vec{v} \in \Ts$, that
    \begin{align*}
        \pairing{\Bb \vec{u}^*,\vec{v}}{\Ts} 
        & = \inner{\partial_t u^*_1 - u^*_2,v_1}_{L^2(Q)} + \Pairing{\partial_t u^*_2 - \Delta_x^{\mathrm{we}} u^*_1,v_2}{L^2(I;H^{-1}(\Omega))}{L^2(I;H^1_0(\Omega))}\\
        & = \inner{f_1,v_1}_{L^2(Q)} + \Pairing{f_2,v_2}{L^2(I;H^{-1}(\Omega))}{L^2(I;H^1_0(\Omega))}= \pairing{\vec{f},\vec{v}}{\Ts}.
    \end{align*}
    Now, regarding the stability estimate, it first holds by \cref{thm:Applications:strong-wave-solution}, together with the 1D Poincaré inequality $\norm{\partial_t f}_{L^2(Q)} \leq \frac{T}{\sqrt{2}}\norm{f}_{L^2(Q)}$ for all $f \in H^1_{0,}(I;L^2(\Omega))$, that 
    \begin{align*}
        \norm{\partial_t u^*_1}_{L^2(Q)}^2 + \norm{u^*_1}_{L^2(I;H^1_0(\Omega))}^2 
        &\leq \hat{C} \norm{\partial_t f_1 + f_2}_{L^2(Q)}^2 
        \leq 2\hat{C}(\norm{\partial_t f_1}_{L^2(Q)}^2 + \norm{f_2}_{L^2(Q)}^2)\\
        &\leq 2\hat{C}\tilde{C}_T \norm{\vec{f}}_{L^2(Q;\R^2)}^2,
    \end{align*}
    with $\tilde{C}_T \deq \max\set{1,\frac{T^2}{2}}$.
    Next, using Poincaré's inequality in $\Omega$, it holds $H^1_0(\Omega)\hookrightarrow_d L^2(\Omega)$ for some $C_\Omega<\infty$, in particular $\norm{f}_{L^2(I;H^{-1}(\Omega))} \leq C_\Omega\norm{f}_{L^2(Q)}$ for all $f\in L^2(Q)$ by \cref{rm: gelfand}, and thus 
    \begin{align*}
        &\norm{\partial_t u^*_2}_{L^2(I;H^{-1}(\Omega))}^2 + \norm{u^*_2}_{L^2(Q)}^2
        = \norm{f_2 + \Delta_x^{\mathrm{we}} u^*_1}_{L^2(I;H^{-1}(\Omega))}^2 + \norm{\partial_t u^*_1 - f_1}_{L^2(Q)}^2\\
        &\qquad \leq 2\parens[\big]{\norm{f_2}_{L^2(I;H^{-1}(\Omega))}^2 + \norm{u^*_1}_{L^2(I;H^1_0(\Omega))}^2 + \norm{\partial_t u^*_1}_{L^2(Q)}^2+\norm{f_1}_{L^2(Q)}^2}\\
        &\qquad \leq (2\tilde{C}_\Omega + 4\hat{C}\tilde{C}_T) \norm{\vec{f}}_{L^2(Q;\R^2)}^2,
    \end{align*}
    with $\tilde{C}_\Omega \deq \max\set{1,C_\Omega^2}$. Now, defining $C_Q^2 \deq 2\tilde{C}_\Omega + 6\hat{C}\tilde{C}_T$, we get
    \begin{align*}
        \norm{\vec{u}^*}_{\Asb}^2
        &= \norm{\partial_t u^*_2}_{L^2(I;H^{-1}(\Omega))}^2 + \norm{\partial_t u^*_1}_{L^2(Q)}^2 + \norm{u^*_1}_{L^2(I;H^1_0(\Omega))}^2 + \norm{u^*_2}_{L^2(Q)}^2
        \\
        &\leq C_Q \norm{\vec{f}}_{L^2(Q;\R^2)}^2.
    \end{align*}
\end{proof}
Although, \eqref{eq:Applications:wave-system-first-order-VF} admits a unique solution $\vec u^*\in \Asb$ for all $\vec f\in H^1_{0,}(I;L^2(\Omega))\times L^2(Q)\subset \Ts'$ we will now show that this can not be extended to all $\vec{f} \in \Ts'$. Thus, for the wave equation, unlike the previous examples, considering $\As$ and $\B$ instead of $\Asb$ and $\Bb$, respectively, is in fact necessary for a well-posed formulation (instead of just the optimal stability) as $\Bb : \Asb \ra \Ts'$ isn't already an isomorphism.
\begin{theorem}\label{lem:Applications:wave-counterexample}
    The operator $\Bb$ in \eqref{eq:Applications:wave-system-first-order-VF} does not define an isomorphism as there does not exists an inf-sup constant $\beta>0$ such that $\beta\norm{\vec u}_{\Asb}\leq \norm{\Bb \vec{u}}_{\Ts'}$ for all $\vec u\in \Asb$, i.e. the inverse of $\Bb$ is not bounded.
\end{theorem}
\begin{proof}
    We construct a counterexample, following the ideas presented in \cite[Theorem 4.2.24]{Zank2019}. To this end, let $\phi_k\in H^1_0(\Omega)$ denote the normalized eigenfunctions and $0<\lambda_1\leq \lambda_2\leq \ldots \to\infty$ the eigenvalues of the spatial Laplacian, given as $-\Delta_x \phi_k = \lambda_k \phi_k$ in $\Omega$ with $\norm{\phi_k}_{L^2(\Omega)}=1$, and consider the function 
    \begin{equation*}
        \vec u^k(t,x) \deq \begin{pmatrix}
           \phi_k(x)\int_0^t s\sin(\sqrt{\lambda_k}s)\, ds \\ \phi_k(x)t\sin(\sqrt{\lambda_k}t)
        \end{pmatrix},
    \end{equation*}
    which solves \eqref{eq:Applications:wave-system-first-order-VF} for $\vec{f}^k \deq (0,f^k)$, $f^k(t,x)\deq 2\phi_k(x)\sin(\sqrt{\lambda_k} t)$. Elementary computations give $\norm{\vec u^k}_{\Asb}^2 
        = \lambda_k^{-3/2}\big(
        T\, \sqrt{\lambda_k}
        +\frac{2}{3}T^3 \lambda_k^{3/2}
        -\frac12 T\, \sin(\sqrt{\lambda_k}T)\big)$
        and 
    $\norm{\vec f^k}_{\Ts'}^2 = \lambda_k^{-3/2}\big( {2T\sqrt{\lambda_k} - \sin(2\sqrt{\lambda_k}T)}\big)$. 
    Therefore, taking the limit we get $\frac{\norm{\vec f^k}_{\Ts'}}{\norm{\vec u^k}_{\Asb}}\to 0$ for $k\to \infty$
    which completes the proof. 
\end{proof}
Thus, in contrast to elliptic and parabolic PDEs, in this case the norm induced by the operator is not equivalent to the norm on $\Asb$. In order to determine $\norm{\Bb\cdot}_{\Ts'}$, using $R_{H^1_0(\Omega)}\equiv -\Delta_x^{\mathrm{we}}$, it holds for all $\vec{u} \in\Asb$, that
\begin{align*}
    &\norm{\Bb \vec{u}}_{\Ts'}^2
    \overset{\eqref{eq:Applications:strong wave:dual Ts}}{=} \norm{\partial_t u_1 - u_2}_{L^2(Q)}^2 + \norm{\partial_t u_2 - \Delta_x^{\mathrm{we}} u_1}_{L^2(I;H^{-1}(\Omega))}^2\\
    &\quad=\norm{\partial_t u_1}_{L^2(Q)}^2 + \norm{u_2}_{L^2(Q)}^2 -2 \inner{\partial_t u_1,u_2}_{L^2(Q)} \\
    &\qquad+ \norm{\partial_t u_2}_{L^2(I;H^{-1}(\Omega))}^2  + \norm{u_1}_{L^2(I;H^1_0(\Omega))}^2 + 2\Pairing{\partial_t u_2, u_1}{L^2(I;H^{-1}(\Omega))}{L^2(I;H^1_0(\Omega))}\\   
    &\quad= \norm{\vec{u}}_{\Asb}^2 + 2\pairing{\partial_{t} \vec{u}, J\vec{u}}{\Ts},
\end{align*}
with $J \deq \begin{psmallmatrix} 0&-\Id\\\Id&0\end{psmallmatrix}$. 
\begin{remark}
Alternatively, using the Riesz operator \eqref{eq: riesz operator} on $\Ts$, given by $R_{\Ts} = \begin{psmallmatrix}\Id & 0\\ 0& -\Delta_x^{\mathrm{we}} \end{psmallmatrix}$, together with \eqref{eq:analysis:abstract supremizer}, we get
\begin{align*}
    \norm{\Bb \vec u}_{\Ts'}^2 & = \pairing{\Bb \vec u,R_{\Ts}^{-1}\Bb \vec u}{\Ts} = \pairing{\partial_t \vec u +A\vec u,R_{\Ts}^{-1}(\partial_t \vec u+A \vec u)}{\Ts} \\
    & = \pairing{\partial_t \vec u, R_{\Ts}^{-1}\partial_t\vec u}{\Ts} + 2\pairing{\partial_t \vec u,R_{\Ts}^{-1}A \vec u}{\Ts} + \pairing{A^* R_{\Ts}^{-1}A\vec u, \vec u}{\mathbb{F}\Ts},
\end{align*}
where $A^* = \begin{psmallmatrix}0&-\Delta_x^{\mathrm{we}}\\-\Id&0\end{psmallmatrix} : W \ra \mathbb{F}W'$ denotes the adjoint of $A$.
To further simplify this expression, note that $R_{\Ts} = AJ^*$ and $J^{-1} = J^*$, hence $R_{\Ts}^{-1}A = J$ as well as $A^* R_{\Ts}^{-1}A= \begin{psmallmatrix}-\Delta_x^{\mathrm{we}} & 0 \\ 0 & \Id\end{psmallmatrix} = R_{\mathbb{F}\Ts},$ with $R_{\mathbb{F}\Ts}$ denoting the Riesz operator \eqref{eq: riesz operator}  on $\mathbb{F}\Ts$. Thus, we get the same result, namely 
\begin{equation}\label{eq:applications:norm induced by wave operator}
\begin{aligned}
    \norm{\Bb \vec u}_{\Ts'}^2
    &= \pairing{\partial_t \vec u, R_{\Ts}^{-1}\partial_t \vec u}{\Ts} + 2\pairing{\partial_t \vec u,J \vec u}{\Ts} + \pairing{R_{\mathbb{F}\Ts}\vec u, \vec u}{\mathbb{F}\Ts}\\
    &= \norm{\partial_t \vec{u}}_{\Ts'}^2 + 2\pairing{\partial_t \vec u,J \vec u}{\Ts} + \norm{\vec{u}}_{\mathbb{F}\Ts}^2
    = \norm{\vec{u}}_{\Asb}^2 + 2\pairing{\partial_t \vec u,J \vec u}{\Ts}.
\end{aligned}
\end{equation}
\end{remark}

As we can see, the representation of the norm $\norm{\Bb \cdot}_{\Ts'}$ is quite similar to the one derived for the heat equation in \eqref{eq:Applications:heat-new-norm-representation}. However, the additional term 
\begin{equation*}
    2\pairing{\partial_{t} \vec{u}, J\vec{u}}{\Ts} = 2\Pairing{\partial_t u_2,u_1}{L^2(I;H^{-1}(\Omega))}{L^2(I;H^1_0(\Omega))} - 2\inner{\partial_t u_1,u_2}_{L^2(Q)}
\end{equation*}
can not be bounded from below by zero. In fact, for $\vec u^k$ as defined in the proof of \cref{lem:Applications:wave-counterexample}, we compute $\pairing{\partial_t \vec u^k,J\vec u^k}{\Ts} = -\frac{T^3}{3}-\frac{\sin(2\sqrt{\lambda_k}T)}{4\lambda_k^{3/2}}+\frac{T}{2\lambda_k}$,
which goes to $-\frac{T^3}{3}<0$ as $k\to\infty$. Moreover, as the function $\vec{u}^k$ solves \eqref{eq:Applications:wave-system-first-order-VF} for the right-hand side $\vec{f}^k$, it holds $\Bb \vec{u}^k = \vec{f}^k$ and we get
\begin{equation*}
    \frac{\norm{\vec f^k}_{\Ts'}}{\norm{\Bb \vec u^k}_{\Ts'}} = \frac{\norm{\vec f^k}_{\Ts'}}{\norm{\vec f^k}_{\Ts'}} = 1\qquad\text{but}\qquad \frac{\norm{\vec f^k}_{\Ts'}}{\norm{\vec u^k}_{\Asb}}\to 0,\quad k\to \infty.
\end{equation*}
Hence, the norms $\norm{\Bb\cdot}_{\Ts'}$ and $\norm{\cdot}_{\Asb}$ are not equivalent. By \eqref{eq: norm X = Bu} and \eqref{eq:applications:norm induced by wave operator}, we have a characterization of $\norm{\cdot}_{\As}$ on $\Asb \subseteq_d\As$.
Now, with the norm $\norm{\cdot}_{\As}$ characterized (at least on $\Asb$), let us consider the space $\As$. Unlike to the elliptic and parabolic equations, we cannot give $\As$ explicitly, but we can characterize it to some extend by \cref{tm: gelfand carrys}.
\begin{corollary}\label{co:applications:adjoint wave well-posed}
    The adjoint problem of \eqref{eq:Applications:wave-system-first-order-VF}, namely to find $\vec{v}^*\in\Ts$ such that
    \begin{equation*}
        \pairing{\Bb^* \vec{v}^*,\vec{u}}{\Asb} 
        = \pairing{\vec{g},\vec{u}}{\Asb}\quad\forall \vec{u}\in\Asb,
    \end{equation*}
    possess a solution for each $\vec{g} \in \Ase \deq L^2(Q) \times H^1_{,0}(I;L^2(\Omega)) \subset_d \Asp$. Further, there exists $C^*_Q<\infty$ independent of $\vec{v}^*$ and $\vec{g}$ such that $\norm{\vec{v}^*}_{\Ts} \leq C^*_Q \norm{\vec{g}}_{L^2(Q;\R^2)}$. In particular \cref{tm: gelfand carrys} holds by \cref{rm: gelfand carrys}.
\end{corollary}
\begin{proof}
 In addition to the flip operator $\mathbb{F}$, we define the time reversal operator by $\mathbb{T}f(t) \deq f(T-t)$ for almost all $t\in I$. It is easy to see, that $\mathbb{F}^2 = \Id$, $\mathbb{T}^2 = \Id$, $\mathbb{F}^* = \mathbb{F}$, $\mathbb{T}^* = \mathbb{T}$, $\mathbb{T}\mathbb{F} = \mathbb{F}\mathbb{T}$, $\partial_t \mathbb{T}\mathbb{F} = - \mathbb{T}\mathbb{F}\partial_t$ and $\mathbb{T}\mathbb{F}A = A^* \mathbb{T}\mathbb{F}$. Further, it holds $\mathbb{T}\mathbb{F} \vec{u} \in \Ts$, $(\mathbb{T}\mathbb{F} \vec{u})(T) = 0$ and $\mathbb{T}\mathbb{F}A\vec{u}\in \mathbb{F}\Ts'$ for all $\vec{u}\in\Asb$.
Now, let $\vec{g} \in \Ase$ be arbitrary but fixed, define $\vec{f}\deq \mathbb{T}\mathbb{F} \vec{g} \in \Tse$ and denote by $\vec{u}^* \in \Asb$ the solution of $\Bb \vec{u}^* = \vec{f}$ in $\Ts'$ given by \cref{co:Applications:wave-system-B1-B2}. Defining $\vec{v}^* \deq \mathbb{T}\mathbb{F} \vec{u}^* \in \Ts$, it holds for all $\vec{u} \in \Asb$, that
    \begin{align*}
        &\pairing{\Bb^* \vec{v}^*,\vec{u}}{\Asb}
        = \pairing{\Bb \vec{u},\vec{v}^*}{\Ts}
        =\pairing{\partial_t\vec{u}+A\vec{u},\vec{v}^*}{\Ts}
        =\pairing{\partial_t\vec{u}+A\vec{u},\mathbb{T}\mathbb{F}\vec{u}^*}{\Ts}
        \\
        &\qquad=\pairing{\mathbb{T}\mathbb{F}\partial_t\vec{u}+\mathbb{T}\mathbb{F}A\vec{u},\vec{u}^*}{\mathbb{F}\Ts}
        =\pairing{-\partial_t\mathbb{T}\mathbb{F}\vec{u}+A^*\mathbb{T}\mathbb{F}\vec{u},\vec{u}^*}{\mathbb{F}\Ts}\\
        &\qquad= \pairing{\partial_t\vec{u}^* + A\vec{u}^*,\mathbb{T}\mathbb{F}\vec{u}}{\Ts} + \inner{\vec{u}^*(T),\underbrace{(\mathbb{T}\mathbb{F}\vec{u})(T)}_{=0}}_{L^2(\Omega)} - \inner{\underbrace{\vec{u}^*(0)}_{=0},(\mathbb{T}\mathbb{F}\vec{u})(0)}_{L^2(\Omega)}\\
        &\qquad= \pairing{\Bb \vec{u}^*,\mathbb{T}\mathbb{F}\vec{u}}{\Ts}
        = \pairing{\vec{f},\mathbb{T}\mathbb{F}\vec{u}}{\Ts}
        = \pairing{\mathbb{T}\mathbb{F} \vec{g},\mathbb{T}\mathbb{F}\vec{u}}{\Ts}
        = \pairing{\vec{g},\vec{u}}{\Asb},
    \end{align*}
    i.e., $\vec{v}^*\in\Ts$ is a solution of the adjoint problem. 
    Thereby we used in the last step, that $\vec{g},\vec{u},\mathbb{T}\mathbb{F} \vec{g},\mathbb{T}\mathbb{F} \vec{u} \in L^2(Q;\R^2)$ and thus $\pairing{\mathbb{T}\mathbb{F} \vec{g},\mathbb{T}\mathbb{F}\vec{u}}{\Ts}
        = \inner{\mathbb{T}\mathbb{F} \vec{g},\mathbb{T}\mathbb{F}\vec{u}}_{L^2(Q;\R^2)}
        = \inner{\vec{g},\vec{u}}_{L^2(Q;\R^2)}
        = \pairing{\vec{g},\vec{u}}{\Asb}$ by \eqref{eq: gelfand pairing = product} and the Gelfand triples $(\Ts,\Tsp,\Ts')$ and $(\Asb,\Asp,\Asb')$. 
    Finally, using the stability estimate for $\vec{u}^*$ provided by \cref{co:Applications:wave-system-B1-B2}, it holds $\norm{\vec{v}^*}_{\Ts} = \norm{\vec{u}^*}_{\mathbb{F}\Ts}
        \leq \norm{\vec{u}^*}_{\Asb} \leq C_Q \norm{\vec{f}}_{L^2(Q;\R^2)}$.
\end{proof}
Let us collect our results.
\begin{corollary}\label{co:applications:collect results wave}
    For the well-posed extension \eqref{eq:well-posed operator-equation} of \cref{ex: strong formulation wave} holds
    \begin{compactenum}[(i)]
        \item $\Asb \hookrightarrow_d \As \hookrightarrow_d L^2(Q; \R^2)$, and the embedding constants read $\norm{u}_{L^2(Q;\R^2)} \leq C_Q \norm{u}_{\As}$, $u\in \As$ and $\norm{u}_{\As} \leq \sqrt{2} \norm{u}_{\Asb}$, $ u\in \Asb$,
            with the constant $C_Q = \sqrt{\max\set{1,C_\Omega^2} + \max\set{3T^2,\frac{3}{2}T^4}}$;
        \item the other directions of the norm inequalities do not holds, i.e. $\norm{\cdot}_{\As}$ is neither equivalent to $\norm{\cdot}_{L^2(Q;\R^2)}$ on $\As$, nor equivalent to $\norm{\cdot}_{\Asb}$ on $\Asb$;
        \item $\Asb \subsetneq \As \subsetneq L^2(Q;\R^2)$;
        \item\label{it:collect results wave - norm representation} for $\vec{u}\in\Asb$, we have $\norm{\vec{u}}_{\As} = \sqrt{\norm{\vec{u}}_{\Asb}^2 + 2\pairing{\partial_t \vec u,J \vec u}{\Ts}}$, where
        $J \deq \begin{pmatrix}
            0 &-\Id  \\ \Id & 0 
        \end{pmatrix}$.
    \end{compactenum}
\end{corollary}
\begin{proof}
The first statement is given by \cref{co:applications:adjoint wave well-posed,rm: notes on gelfand carrys} implying $C^*_Q \equiv C_Q$, together with the definition of $C_Q$ and $\hat{C}$ in \cref{co:Applications:wave-system-B1-B2,rm:Applications:strong-wave-solution:constant}, respectively, while $\norm{\Bb}_{\cL(\Asb,\Ts')} \leq \sqrt{2}$ holds by \cref{lm:application:wave operator bounded}. The last statement was shown in \eqref{eq:applications:norm induced by wave operator}. Further, that $\norm{\cdot}_{\As}$ is not equivalent to $\norm{\cdot}_{\Asb}$ follows by \cref{lem:Applications:wave-counterexample} and \eqref{eq: norm X = Bu}. Next, if $\norm{\cdot}_{\As}$ would be equivalent to $\norm{\cdot}_{L^2(Q;\R^2)}$, we would get $\As = L^2(Q;\R^2)$ since $\As$ is dense in $L^2(Q;\R^2)$ by (i) but also complete w.r.t. $\norm{\cdot}_{\As}$. Since $\B:\As \ra \Ts'$ is an isomorphism, we would end up with $L^2(Q;\R^2)$ being isomorph to $\Ts' = L^2(I;L^2(\Omega)\times H^{-1}(\Omega))$, which is not the case as $L^2(\Omega)$ is not isomorph to $H^{-1}(\Omega)$. Thus, we have also shown, that $\As \neq L^2(Q;\R^2)$ has to hold. Finally, by the construction of $\As$, it would follow from $\Asb = \As$, that $\range(\Bb) = \Ts'$, i.e. $\Bb \in \Lis(\Asb,\Ts')$ by the bounded inverse theorem/open mapping theorem, \crefl{it: injective} and \cref{lm:application:wave operator bounded}. This is a contradiction to \cref{lem:Applications:wave-counterexample}, and we conclude $\Asb \neq \As$.
\end{proof}
\begin{remark}
    By \cref{rm: elliptic result hold for general operator}, we can replace $-\Delta_x^{\mathrm{we}}$ by any bounded self-adjoint coercive operator $A_x : H^1_0(\Omega) \ra H^{-1}(\Omega)$ provided that $A_x$ satisfies the assumptions of \cite[Ch.\,3 Thm.\,8.1]{LionsM1972} (i.e. \cref{thm:Applications:strong-wave-solution} holds true), by replacing the norm on $H^1_0(\Omega)$ by $\norm{A_x\cdot}_{H^{-1}(\Omega)}$. 
\end{remark}

\subsubsection{Weak in time}
\label{ch:applications:weak in time wave}
We consider the weak in time variational formulation of the wave equation as given in \cref{ex: weak formulation wave}.
\begin{remark}
    The operator $\Bb$ in \eqref{eq:applications:sym-wave:VF} is bounded, i.e. $\Bb\in\cL(\Asb,\Ts')$ as it holds for $u\in\Asb$ and $v\in\Ts$ by Young's inequality\footnote{$
            ab + cd
            = \sqrt{(ab + cd)^2}
            \leq \sqrt{a^2b^2 + a^2d^2 + b^2c^2 + c^2d^2}
            = \sqrt{a^2 + c^2} \sqrt{b^2 + d^2}$  for all $a,b,c,d\geq 0$.},
    that
    $\abs{\pairing{\Bb u,v}{\Ts}}
        \leq \norm{\partial_t u}_{L^2(Q)} \norm{\partial_t v}_{L^2(Q)} + \norm{\nabla_x u}_{L^2(Q)} \norm{\nabla_x v}_{L^2(Q)}
        = \norm{u}_{\Asb}\norm{v}_{\Ts}$.
\end{remark}
\begin{theorem}\label{thm: weak wave solution for L2}
    The variational problem \eqref{eq:applications:sym-wave:VF} possess a unique solution $u^*\in\Asb$ for all $f\in L^2(Q)$ and the solution satisfies $\norm{u^*}_{H^1_0(\Omega)} \leq \frac{T}{\sqrt{2}}\norm{f}_{L^2(Q)}$.
\end{theorem}
\begin{proof}
    This is a well known result, see e.g. \cite[Theorem 5.1]{SteinZ2020} for this exact statement or \cite[ch.\,IV Theorems 3.1\,\&\,3.2, eq.\,(3.17)]{Ladyz1985} for a more general setting.
\end{proof}
\begin{remark}
    By the above existence result, \cref{tm:abstract setting well-posed} holds by \cref{lm: solution of equation sufficient for B1 B2,rm: solution of equation sufficient for B1 B2}, while \cref{tm: gelfand carrys} holds as $\Bb$ is symmetric and thus self-adjoint (except for a switch of the initial/terminal conditions encoded in $\Asb$ and $\Ts$, but $\Asb$ and $\Ts$ are isometric isomorphic and thus interchangeable).
\end{remark}
As for the strong in time formulation, the above existence result cannot be generalized to all right-hand sides $f\in \Ts'$ as the operator is not inf-sup stable.
\begin{theorem}[{\cite[Theorem 1.1]{SteinZ2022}}]
    The operator $\Bb$ in \eqref{eq:applications:sym-wave:VF} does not define an isomorphism. In particular, there does not exist an inf-sup constant $\beta>0$ such that $\beta\norm{u}_{\Asb}\leq \norm{\Bb u}_{\Ts'}$ for all $u\in \Asb$, i.e., the inverse of $\Bb$ is not bounded.
\end{theorem}
Now, denoting the well-posed and optimally stable extension of $\Asb$ and $\Bb$ by $\As$ and $\B$, respectively, we get the following result, which is proven analogously to \cref{co:applications:collect results wave}.
\begin{corollary}
    For the well-posed extension \eqref{eq:well-posed operator-equation} of \eqref{eq:applications:sym-wave:VF} holds
    \begin{compactenum}[(i)]
        \item $\Asb = H^1_{0,}(I;H^1_0(\Omega)) \hookrightarrow_d \As \hookrightarrow_d L^2(Q)$ and the embedding constants read $\norm{u}_{L^2(Q)} \leq \tfrac{T}{\sqrt{2}} \norm{u}_{\As}$, $u\in \As$ and $\norm{u}_{\As} \leq \norm{u}_{H^1_0(Q)}$, $u\in \Asb$;
        \item the other directions of the norm inequalities do not holds, i.e. $\norm{\cdot}_{\As}$ is neither equivalent to $\norm{\cdot}_{L^2(Q)}$ on $\As$, nor equivalent to $\norm{\cdot}_{H^1_0(Q)}$ on $\Asb$;
        \item $\Asb \subsetneq \As \subsetneq L^2(Q)$, i.e. $\As$ is a strict subset of $L^2(Q)$ and a strict superset of $\Asb$.
    \end{compactenum}
\end{corollary}

\begin{remark}
    Although constructed by completely different means, the well-posed and optimally stable extension \eqref{eq:well-posed operator-equation} of the weak wave equation in \cref{ex: weak formulation wave} corresponds (up to isometric isomorphism) exactly with the well-posed and optimally stable formulation of the wave equation introduced in \cite{SteinZ2022}, with $\As \equiv \mathcal{H}_{0;0,}$ and $\B \equiv \mathcal{E}^*\square\widetilde{(\cdot)}$. This becomes clear by the uniqueness of the completion and continuous extension as stated in \cref{rm: notes on extended definition}. Moreover, our framework also covers other formulations, e.g., \cite{FuhrerGK2025}, where the wave equation is reformulated as a first order system in the velocity and the flux variable. 
\end{remark}
\begin{remark}
    By \cref{rm: elliptic result hold for general operator}, we can replace $-\Delta_x^{\mathrm{we}}$ by any bounded self-adjoint coercive operator $A_x : H^1_0(\Omega) \ra H^{-1}(\Omega)$ provided that $A_x$ satisfies the assumptions of \cite[ch.\,IV Theorems 3.1\,\&\,3.2]{Ladyz1985} (i.e. \cref{thm: weak wave solution for L2} holds true), by replacing the norm on $H^1_0(\Omega)$ by $\norm{A_x\cdot}_{H^{-1}(\Omega)}$. 
\end{remark}

\subsubsection{Ultra-weak in time}
\label{ch:applications:ultraweak in time wave}
We consider the ultra-weak in time variational formulation of the wave equation as given in \cref{ex: ultra-weak formulation wave}. 
As this setting can be reduced to the strong in time formulation given in \cref{ex: strong formulation wave}, we denote by $\Bb,\Asb$ and $\Ts$ the strong in time operator, trial space and test space, respectively, as given in \cref{ex: strong formulation wave} and denote the ultra-weak in time operator, trial space and test space as given in \cref{ex: ultra-weak formulation wave} by $\Bb^{\mathrm{uw}}$, $\Asb^{\mathrm{uw}}$ and $\Ts^{\mathrm{uw}}$, respectively. Using the flip operator $\mathbb{F}\begin{psmallmatrix}x_1\\x_2\end{psmallmatrix} \deq \begin{psmallmatrix}x_2\\x_1\end{psmallmatrix}$ and the time reversal operator $\mathbb{T}f(t) \deq f(T-t)$ as introduced in \S\ref{ch:applications:strong in time wave}, we get
\begin{equation*}
    \Asb^{\mathrm{uw}} = \mathbb{TF}\Ts = \mathbb{F}\Ts,\quad
    \Ts^{\mathrm{uw}} \deq \mathbb{TF}\Asb,\quad
    \Pairing{\Bb^{\mathrm{uw}} \vec{u},\vec{v}}{[\Ts^{\mathrm{uw}}]'}{\Ts} = \Pairing{\vec{u},-\partial_t \vec{v}+A^* \vec{v}}{\mathbb{F}\Ts}{\mathbb{F}\Ts'},
\end{equation*}
for all $\vec{u} \in\Asb^{\mathrm{uw}}$ and all $\vec{v}\in \Ts^{\mathrm{uw}}$.
Using $-\mathbb{TF}\partial_t = \partial_t \mathbb{TF}$, $\mathbb{TF}A^* = A\mathbb{TF}$ and $(\mathbb{TF})^* = \mathbb{TF}$ as stated in the proof of \cref{co:applications:adjoint wave well-posed}, it holds for all $\vec{u}\in \Asb^{\mathrm{uw}}$ and all $\vec{v} \in \Ts^{\mathrm{uw}}$, that
\begin{align*}
    &\pairingb{\Bb^{\mathrm{uw}} \vec{u},\vec{v}}{\Ts^{\mathrm{uw}}}
    = \pairing{-\partial_t\vec{v}+A^*\vec{v},\vec{u}}{\mathbb{TF}\Ts}
    = \pairing{-\mathbb{TF}\partial_t\vec{v}+\mathbb{TF}A^*\vec{v},\mathbb{TF}\vec{u}}{\Ts}\\
    &\qquad= \pairing{\partial_t\mathbb{TF}\vec{v}+A\mathbb{TF}\vec{v},\mathbb{TF}\vec{u}}{\Ts}
    = \pairing{\Bb\mathbb{TF}\vec{v},\mathbb{TF}\vec{u}}{\Ts}
    = \pairing{\Bb^*\mathbb{TF}\vec{u},\mathbb{TF}\vec{v}}{\Asb}\\
    &\qquad= \pairingb{\mathbb{TF}\Bb^*\mathbb{TF}\vec{u},\vec{v}}{\Ts^{\mathrm{uw}}}.
\end{align*}
Thus, it holds $\Bb^{\mathrm{uw}} = \mathbb{TF}\Bb^*\mathbb{TF}$ and by a similar calculation $(\Bb^{\mathrm{uw}})^* = \mathbb{TF}\Bb\mathbb{TF}$. With that, we can use \cref{lm:application:wave operator bounded,co:Applications:wave-system-B1-B2,lem:Applications:wave-counterexample,co:applications:adjoint wave well-posed} to state similar results regarding $\Bb^{\mathrm{uw}}$ and $(\Bb^{\mathrm{uw}})^*$, noting, that $\Bb^{\mathrm{uw}}$ is bounded / isomorphic if and only if $(\Bb^{\mathrm{uw}})^*$ is bounded / isomorphic.
The only result in \S\ref{ch:applications:strong in time wave}, that cannot be reproduced for $\Bb^{\mathrm{uw}}$ is the norm representation \eqref{eq:applications:norm induced by wave operator} as we do not have an explicit formula for the dual norm or for the inverse of the Riesz operator on $\Asb$ (and hence on $\Ts^{\mathrm{uw}} = \mathbb{TF}\Asb$) as it was the case in \eqref{eq:Applications:strong wave:dual Ts}. Thus, we are not able to derive a representation formula for $\norm{\Bb^{\mathrm{uw}}\cdot}_{[\Ts^{\mathrm{uw}}]'}$, neither by a direct calculation nor by using \eqref{eq:analysis:abstract supremizer}. 
Now, denoting the well-posed and optimally stable extension  of $\Asb^{\mathrm{uw}}$ and $\Bb^{\mathrm{uw}}$ by $\As^{\mathrm{uw}}$ and $\B^{\mathrm{uw}}$, respectively, \cref{co:applications:collect results wave} holds true after replacing every $\Asb$ and $\As$ by $\Asb^{\mathrm{uw}}$ and $\As^{\mathrm{uw}}$, respectively, except for the norm representation (\ref{it:collect results wave - norm representation}).

%% file: sections/conclusion.tex
\section{Conclusions and Outlook}
\label{sec:conclusions}
In this paper, we presented a general abstract framework towards well-posed and (optimally) stable formulations of linear operator equations. The starting point is always a classical formulation, which does not need to be well-posed. This requires a Gelfand triple structure for the test space. The second step is to restrict the space for the right-hand sides in such a way that the operator equation admits a solution for such right-hand side data (think of smooth functions). Finally, we form a completion for the trial space and a unique continuous extension of the operator. This procedure can be made more explicit if also the trial space allows for a Gelfand triple structure.

This general framework is applied to the Poisson, heat and wave equation, the latter two in a variational space-time setting. Our findings are summarized in Table \ref{Tab:Summary}. We reproduce well-known results concerning strong, weak and ultra-weak formulations of elliptic and parabolic problems. However, the presented setting also applies for the hyperbolic wave equation, where we derive well-posed formulations, which are (to the very best of our knowledge) new. We can characterize trial spaces and induced norms in such a way that the formulation of the wave equation in these spaces is well-posed and optimally stable. It turns out that these are non-standard Sobolev-type spaces.
\begin{table}[!htb]
\renewcommand{\arraystretch}{2}
\footnotesize
\begin{center}
\begin{tabular}{l|l|l|l|l|l}
    Eq. & Form & $U$ & $V$ & $\bar{U}$ & Ref.\ \\ \hline\hline
    Poisson 
    & strong     & $H^\Delta \cap H^1_0$ & $L^2$ & $U$ 
    & \S\ref{ch:applications:strong Poisson} \\ \cline{2-6}
    & weak       & $H^1_0$ & $H^1_0$ & $U$ 
    & \S\ref{ch:applications:weak Poisson}\\ \cline{2-6}
    & ultra-weak & $L^2$ & $H^\Delta \cap H^1_0$ & $U$ 
    & \S\ref{ch:applications:utraweak Poisson}\\ \hline
    heat
    & strong in $t$ & \parbox{1.85cm}{$L^2(I;H^1_0)$ \\ $\cap\, H^1_{0,}(I;H^{-1})$} & $L^2(I;H^1_0)$ & $U$ 
    & \S\ref{ch:applications:strong in time heat} \\ \cline{2-6}
    & weak in $t$   & $L^2(I;H^1_0)$ & \parbox{2.0cm}{$L^2(I;H^1_0)$ \\ $\cap\, H^1_{,0}(I;H^{-1})$} & $U$ 
    & \S\ref{ch:applications:weak in time heat} \\ \hline
    wave
    & \parbox{1.6cm}{strong in $t$\\ (1$^{\text{st}}$ or.)}& \parbox{2.2cm}{
    $H^1_{0,}(I; L^2\times H^{-1})$ \\ $\cap\, L^2(I;H^1_0 \times L^2)$} & $L^2(I;L^2\times H^1_0)$ 
    & \parbox{2.6cm}{$\Asb\!\hookrightarrow_d\!\As\!\hookrightarrow_d L^2(Q)^2$\\
    $\Asb \subsetneq \As \subsetneq L^2(Q)^2$ }
    & \S\ref{ch:applications:strong in time wave}
    \\ \cline{2-6}
    & weak in $t$   & $H^1_{0,}(I;H^1_0)$ & $H^1_{,0}(I; H^1_0)$
    & \parbox{2.6cm}{$\Asb\!\hookrightarrow_d\!\As\!\hookrightarrow_d L^2(Q)$\\
    $\Asb \subsetneq \As \subsetneq L^2(Q)$ }    
    & \S\ref{ch:applications:weak in time wave} \\ \cline{2-6}
    & \parbox{1.6cm}{ultra-weak in $t$ (1$^{\text{st}}$ or.)} 
    & $L^2(I;H^1_0\times L^2)$ 
    & \parbox{2.2cm}{
    $H^1_{,0}(I; H^{-1}\times L^2)$ \\ $\cap\, L^2(I;L^2\times H^1_0)$} 
    & \parbox{2.6cm}{$\Asb\!\hookrightarrow_d\!\As\!\hookrightarrow_d L^2(Q)^2$\\
    $\Asb \subsetneq \As \subsetneq L^2(Q)^2$}
    & \S\ref{ch:applications:ultraweak in time wave} \\ \hline
\end{tabular}
\caption{\label{Tab:Summary}Summary of the application of the general framework to the described examples. We omit the dependency on $\Omega$ for brevity.}
\end{center}
\end{table}
\normalsize

As already indicated above, this paper is meant to lay the theoretical foundation for a subsequent numerical discretization (in terms of Galerkin and Petrov-Galerkin schemes) and also for model reduction of parameterized linear operator equations. This will be the topic of subsequent parts of this work. It is clear that the numerical realization of the involved norms will be a challenge. Concerning model reduction, we will investigate to which extend the combination of parameter-dependent trial spaces (i.e., leaving the realm of linear model reduction and the known barrier of the Kolmogorov $n$-width for transport- and wave-type problems) and parameter-independent test spaces (allowing for an efficient computation of the norm of the residual as an a posteriori error estimator) might be beneficial.

Moreover, we are aiming to apply the presented framework also to first order transport problems, singular integral  and Schrödinger-type operators.

%% file: paper.bbl
\begin{thebibliography}{10}

\bibitem{BabuskaJanik}
I.~Babu\v{s}ka and T.~Janik.
\newblock The {$h$}-{$p$} version of the finite element method for parabolic
  equations. {I}. {T}he {$p$}-version in time.
\newblock {\em Num.\ Meth.\ PDEs}, 5(4):363--399, 1989.

\bibitem{BoyerF2013}
F.~Boyer and P.~Fabrie.
\newblock {\em Mathematical Tools for the Study of the Incompressible
  Navier-Stokes Equations and Related Models}.
\newblock Springer New York, 2013.

\bibitem{Brezi2011}
H.~Brezis.
\newblock {\em Functional Analysis, Sobolev Spaces and Partial Differential
  Equations}.
\newblock Springer New York, 2011.

\bibitem{CheginiStevenson}
N.~Chegini and R.~Stevenson.
\newblock Adaptive wavelet schemes for parabolic problems: sparse matrices and
  numerical results.
\newblock {\em SIAM J. Numer. Anal.}, 49(1):182--212, 2011.

\bibitem{DahmeHSW2012}
W.~Dahmen, C.~Huang, C.~Schwab, and G.~Welper.
\newblock Adaptive {P}etrov-{G}alerkin methods for first order transport
  equations.
\newblock {\em {SIAM} J.\ Numeri.\ Anal.}, 50(5):2420--2445, Jan. 2012.

\bibitem{DautrayLions}
R.~Dautray and J.-L. Lions.
\newblock {\em Mathematical analysis and numerical methods for science and
  technology. {V}ol. 5}.
\newblock Springer-Verlag, Berlin, 1992.
\newblock Evolution problems. I.

\bibitem{EinsiW2017}
M.~Einsiedler and T.~Ward.
\newblock {\em Functional Analysis, Spectral Theory, and Applications}.
\newblock Springer International Publishing, 2017.

\bibitem{ErnGuermond}
A.~Ern and J.-L. Guermond.
\newblock {\em Finite elements {III}---first-order and time-dependent {PDE}s},
  volume~74 of {\em Texts in Applied Mathematics}.
\newblock Springer, Cham, 2021.

\bibitem{FuhrerGK2025}
T.~F\"{u}hrer, R.~Gonz\'{a}lez, and M.~Karkulik.
\newblock Well-posedness of first-order acoustic wave equations and space-time
  finite element approximation.
\newblock {\em IMA J.\ Numer.\ Anal.}, 2025.

\bibitem{FuhrerK2021}
T.~F\"{u}hrer and M.~Karkulik.
\newblock Space--time least-squares finite elements for parabolic equations.
\newblock {\em Comp.\ \& Math.\ Appl.}, 92:27--36, 2021.

\bibitem{GesztesyMitrea2011}
F.~Gesztesy and M.~Mitrea.
\newblock A description of all self-adjoint extensions of the {L}aplacian and
  {K}re\u in-type resolvent formulas on non-smooth domains.
\newblock {\em J. Anal. Math.}, 113:53--172, 2011.

\bibitem{GilbaT2001}
D.~Gilbarg and N.~S. Trudinger.
\newblock {\em Elliptic Partial Differential Equations of Second Order}.
\newblock Springer Berlin Heidelberg, 2001.

\bibitem{Hadamard}
J.~Hadamard.
\newblock Sur les probl\`{e}mes aux d\'{e}riv\'{e}s partielles et leur
  signification physique.
\newblock {\em Princeton University Bulletin}, 13:49--52, 1902.

\bibitem{HytoeNVW2016}
T.~Hyt\"{o}nen, J.~van Neerven, M.~Veraar, and L.~Weis.
\newblock {\em Analysis in Banach Spaces}.
\newblock Springer International Publishing, 2016.

\bibitem{Ladyz1985}
O.~A. Ladyzhenskaya.
\newblock {\em The Boundary Value Problems of Mathematical Physics}.
\newblock Springer New York, 1985.

\bibitem{LionsM1972}
J.~L. Lions and E.~Magenes.
\newblock {\em Non-Homogeneous Boundary Value Problems and Applications}.
\newblock Springer Berlin Heidelberg, 1972.

\bibitem{SchwabStevenson}
C.~Schwab and R.~Stevenson.
\newblock Space-time adaptive wavelet methods for parabolic evolution problems.
\newblock {\em Math.\ Comp.}, 78(267):1293--1318, 2009.

\bibitem{Steinbach2015}
O.~Steinbach.
\newblock Space-time finite element methods for parabolic problems.
\newblock {\em Comput. Methods Appl. Math.}, 15(4):551--566, 2015.

\bibitem{SteinZ2020}
O.~Steinbach and M.~Zank.
\newblock Coercive space-time finite element methods for initial boundary value
  problems.
\newblock {\em {ETNA} - Electr.\ Trans.\ Numer.\ Anal.}, 52:154--194, 2020.

\bibitem{SteinZ2022}
O.~Steinbach and M.~Zank.
\newblock A generalized inf{--}sup stable variational formulation for the wave
  equation.
\newblock {\em J.\ Math.\ Anal.\ Appl.}, 505(1):24, 01 2022.

\bibitem{TantardiniVeeser2016}
F.~Tantardini and A.~Veeser.
\newblock The {$L^2$}-projection and quasi-optimality of {G}alerkin methods for
  parabolic equations.
\newblock {\em SIAM J. Numer. Anal.}, 54(1):317--340, 2016.

\bibitem{SpaceTimeUrbanPatera1}
K.~Urban and A.~Patera.
\newblock A new error bound for reduced basis approximation of parabolic
  partial differential equations.
\newblock {\em C.R.\ Math.\ Acad.\ Sci.\ Paris}, 350(3-4):203--207, 2012.

\bibitem{UrbanP2013}
K.~Urban and A.~Patera.
\newblock An improved error bound for reduced basis approximation of linear
  parabolic problems.
\newblock {\em Math.\ Comp.}, 83(288):1599--1615, 2013.

\bibitem{Zank2019}
M.~Zank.
\newblock {\em Inf-Sup Stable Space-Time Methods for Time-Dependent Partial
  Differential Equations}.
\newblock Ph{D} thesis, TU Graz, Institute of Applied Mathematics, 2019.

\end{thebibliography}
